\documentclass[11pt]{article}
\usepackage{e-jc}
\usepackage{amsthm,amsmath,amssymb}

\usepackage[colorlinks=true,citecolor=black,linkcolor=black,urlcolor=blue]{hyperref}

\usepackage{graphicx}

\theoremstyle{plain}
\newtheorem{thm}{Theorem}[section]

\newtheorem{lem}[thm]{Lemma}

\theoremstyle{definition}

\theoremstyle{remark}
\newtheorem{rmk}{Remark}

\numberwithin{equation}{section}
\numberwithin{figure}{section}
\numberwithin{table}{section}

\newcommand{\wt}{\operatorname{wt}}

\newcommand{\M}{\operatorname{M}}

\newcommand{\up}{\operatorname{up}}
\newcommand{\down}{\operatorname{down}}
\newcommand{\level}{\operatorname{level}}
\newcommand{\area}{\operatorname{area}}

\DeclareMathOperator{\AR}{AR}

\title{Generating function of the tilings of Aztec rectangle with holes}
\author{Tri Lai\footnote{This research was supported in part by the Institute for Mathematics and its Applications with funds provided by the National Science Foundation.}\\
Institute for Mathematics and its Applications\\
Minneapolis, MN 55455\\
\texttt{tmlai@ima.umn.edu}}

\date{\small Mathematics Subject Classifications: 05A15,  05B45}

\begin{document}

\maketitle

\begin{abstract}
We consider a generating function of the domino tilings of an Aztec rectangle with several boundary unit squares removed. Our generating function involves two statistics: the rank of the tiling and half number of vertical dominoes as in the Aztec diamond theorem by Elkies, Kuperberg, Larsen and Propp. In addition, our work deduces a combinatorial explanation for an interesting connection between the number of lozenge tilings of a semihexagon and the number of domino tilings of an Aztec rectangle.

\end{abstract}

\section{Introduction}
A lattice partitions the plane into fundamental regions. A (lattice) \textit{region} considered in this paper is a finite connected union of fundamental regions. A \textit{tile} is the union of any two fundamental regions sharing an edge. A \textit{tiling} of a region is a covering of the region by tiles so that there are no gaps or overlaps.

The \textit{Aztec diamond} of order $n$ is the union of all unit squares inside the  contour $|x|+|y|= n+1$ (see Figure \ref{aztecdiamond} for several first Aztec diamonds).  For each  (domino) tiling $T$ of the Aztec diamond, we denote by $v(T)$ haft number of vertical dominoes in $T$, and $r(T)$ the \textit{rank} of $T$ that is defined as follows. The \emph{minimal tiling} $T_0$ consisting of all horizontal dominoes has rank 0; and the rank $r(T)$ of $T$ is the minimal number of \emph{elementary moves} required to reach $T$ from $T_0$ (see Figure \ref{Drawdominob}(a) for two types of the elementary moves, and Figures \ref{Drawdominob}(b)--(e) for several domino tilings of the Aztec diamond of order 2 together with their ranks).
\begin{figure}\centering
\includegraphics[width=12cm]{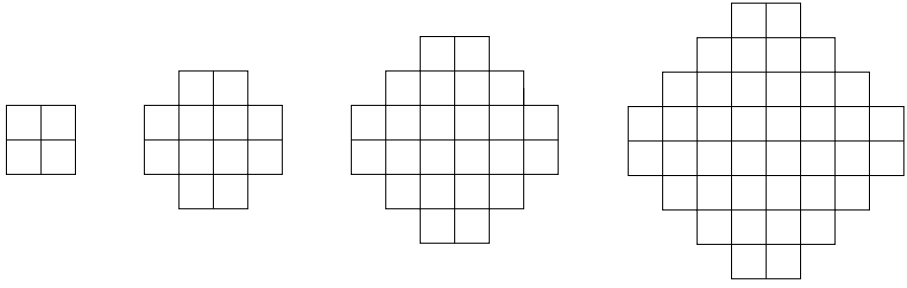}
\caption{From left to right, the Aztec diamonds of order $1$, $2$, $3$ and $4$. }
\label{aztecdiamond}
\end{figure}

\begin{figure}\centering
\includegraphics[width=10cm]{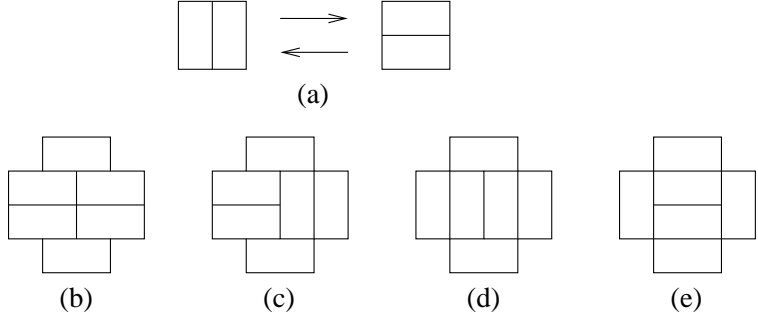}
\caption{(a) The elementary moves: rotations of a $2\times 2$ block of two vertical or horizontal dominoes. (b) The minimal tiling of the Aztec diamond of order $2$. (c) A tiling of rank 1. (d) A tiling of rank 2. (e) A tiling of rank 3.}
\label{Drawdominob}
\end{figure}

Elkies, Kuperberg, Larsen and Propp \cite{Elkies} proved a simple product formula for the generating function of the tilings of an Aztec diamond.
\begin{thm}[Aztec Diamond Theorem \cite{Elkies}]\label{Aztecthm}
For positive integer $n$
\begin{equation}
\sum_{T}q^{r(T)}t^{v(T)}=\prod_{k=0}^{n-1}(1+tq^{2k+1})^{n-k},
\end{equation}
where the sum is taken over all tilings $T$ of the Aztec diamond region of order $n$.
\end{thm}
The $t=q=1$ specialization of the Aztec Diamond Theorem implies that the number of tilings of the Aztec diamond of order $n$ is equal to $2^{n(n+1)/2}$. Besides the fours original proofs in \cite{Elkies}, a number of further proofs of the Aztec Diamond Theorem and its special cases have been given by several authors (see e.g. \cite{Bosio}, \cite{Brualdi}, \cite{Eu}, \cite{Kamioka}, \cite{Kuo}, \cite{propp}). Moreover, we proved in \cite{Tri2} a generalization of the above unweighted Aztec Diamond Theorem for a family of $4$-vertex regions on the square lattice with diagonals drawn in.

\begin{figure}\centering
\includegraphics[width=14cm]{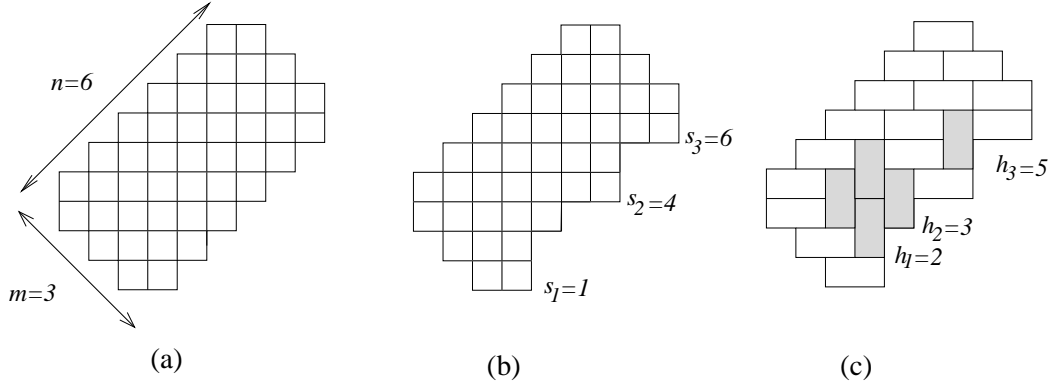}
\caption{(a) The Aztec rectangle $\mathcal{AR}_{3,6}$. (b) The Aztec rectangle with holes $\mathcal{AR}_{3,6}(1,4,6)$. (c) The minimal tiling of $\mathcal{AR}_{3,6}(1,4,6)$. }
\label{Aztecregion}
\end{figure}

The Aztec rectangle is a natural generalization of the Aztec diamond. Figure \ref{Aztecregion}(a) shows an example of the Aztec rectangle.  Denote by $\mathcal{AR}_{m,n}$ the Aztec rectangle having $m$ unit squares on the southwest side and $n$ unit squares along the northwest side.  For $m<n$, $\mathcal{AR}_{m,n}$ does not have any tiling, however when we remove $n-m$ unit squares along the southeast side, the number of tilings of the resulting region is given by a simple product formula (see e.g. \cite{Gessel}, Lemma 3). Denote by $\mathcal{AR}_{m,n}(s_1,s_2,\dotsc,s_{m})$ the $(m\times n)$-Aztec rectangle, where all unit squares on the southeast side, except for the $s_1$-st, the $s_2$-nd, $\dots$ and the $s_m$-th ones, have been removed (see Figure \ref{Aztecregion}(b) for an example). We call the unit squares, which have been removed, \emph{holes}, and our region an \emph{Aztec rectangle with holes}.

In general, an Aztec rectangle with holes does not admit a tiling consisting of all horizontal dominoes.  Assume that $\{h_1,\dotsc,h_{n-m}\}$ ($=\{1,2,\dotsc,n\}-\{s_1,\dotsc,s_{m}\}$) is the position set of the holes. We (re-)define our minimal tiling (still denoted by $T_0$) as follows: Next to the hole at the position $h_i$ on the southeast side, we place southeast-to-northwest strip of $m-(h_i-i)$ vertical dominoes, for $i=1,2,\dotsc,n-m$, and cover the rest of the region by horizontal dominoes. Figure \ref{Aztecregion}(c) illustrates the minimal tiling of the region $\mathcal{AR}_{3,6}(1,4,6)$. We define two statistics $r$ and $v$ for an Aztec rectangle with holes in the same way as the case of the Aztec diamonds.

We consider the following tiling generating function
\begin{equation}
F(q,t):=\sum_{T}q^{r(T)}t^{v(T)},
\end{equation}
where the sum is taken over all tilings $T$ of $\mathcal{AR}_{m,n}(s_1,s_2,\dotsc,s_{m})$. The main result of our paper is the following theorem.

\begin{thm}\label{main}
Assume $m,n,s_1,s_2,\dotsc,s_m$ are positive integers, so that $m<n$ and $1\leq s_1<s_2<\dotsc<s_m\leq n$. Then the tiling generating function of $\mathcal{AR}_{m,n}(s_1,s_2,\dotsc,s_{m})$ is given by
\begin{align}
F(q,t)&=q^{\frac{2(m-1)m(m+1)}{3}+2\sum_{i=1}^{m}(s_i-i)-\sum_{1\leq i\leq j  \leq m}2(s_{i}+j-i-1)}\\
&\quad\quad\quad\times\prod_{k=0}^{m-1}(1+tq^{2k+1})^{m-k}\prod_{1\leq i<j\leq m}\frac{q^{2s_j}-q^{2s_i}}{q^{2j}-q^{2i}}.
\end{align}
\end{thm}

\section{Subgraph replacements}

The \emph{dual graph} of a region $R$ is the graph whose vertices are the fundamental regions of $R$ and whose edges connect precisely two fundamental regions sharing an edge. A \emph{perfect matching} of a graph $G$ is a collection of disjoint edges covering all vertices of $G$.  The tilings of a region can be identified naturally with the perfect matchings of its dual graph.

Let $G$ be a weighted graph. The  \emph{matching generating function} $\M(G)$ of $G$  is defined to be the sum of weights of all perfect matchings of $G$, where the \emph{weight} of a perfect matching is the product of weights of its constituent edges.  If the tiles of a region $R$ carry some weights, we define similarly the \emph{tiling generating function} $\M(R)$ of the region $R$.  Moreover, each edge of the dual graph $G$ of the region $R$ carries the same weight as its corresponding tile in $R$.

\medskip

\begin{figure}\centering
\begin{picture}(0,0)%
\includegraphics{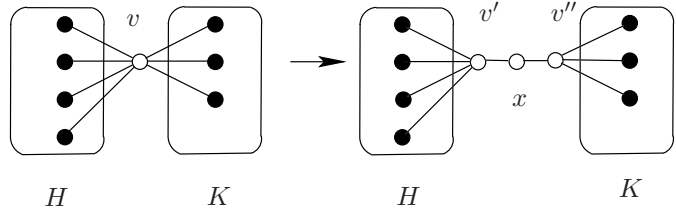}%
\end{picture}%
\setlength{\unitlength}{3947sp}%
\begingroup\makeatletter\ifx\SetFigFont\undefined%
\gdef\SetFigFont#1#2#3#4#5{%
  \reset@font\fontsize{#1}{#2pt}%
  \fontfamily{#3}\fontseries{#4}\fontshape{#5}%
  \selectfont}%
\fi\endgroup%
\begin{picture}(4188,1361)(593,-556)
\put(1336,591){\makebox(0,0)[lb]{\smash{{\SetFigFont{10}{12.0}{\familydefault}{\mddefault}{\updefault}{$v$}%
}}}}
\put(3549,621){\makebox(0,0)[lb]{\smash{{\SetFigFont{10}{12.0}{\familydefault}{\mddefault}{\updefault}{$v'$}%
}}}}
\put(3757, 89){\makebox(0,0)[lb]{\smash{{\SetFigFont{10}{12.0}{\familydefault}{\mddefault}{\updefault}{$x$}%
}}}}
\put(3999,621){\makebox(0,0)[lb]{\smash{{\SetFigFont{10}{12.0}{\familydefault}{\mddefault}{\updefault}{$v''$}%
}}}}
\put(820,-541){\makebox(0,0)[lb]{\smash{{\SetFigFont{10}{12.0}{\familydefault}{\mddefault}{\updefault}{$H$}%
}}}}
\put(1840,-535){\makebox(0,0)[lb]{\smash{{\SetFigFont{10}{12.0}{\familydefault}{\mddefault}{\updefault}{$K$}%
}}}}
\put(3031,-535){\makebox(0,0)[lb]{\smash{{\SetFigFont{10}{12.0}{\familydefault}{\mddefault}{\updefault}{$H$}%
}}}}
\put(4426,-484){\makebox(0,0)[lb]{\smash{{\SetFigFont{10}{12.0}{\familydefault}{\mddefault}{\updefault}{$K$}%
}}}}
\end{picture}%
\caption{Vertex splitting.}
\label{vertexsplitting}
\end{figure}

Next, we present several preliminary results of the subgraph replacement method.

\begin{lem} [Vertex-Splitting Lemma]\label{VS}
 Let $G$ be a weighted graph and $v$ a vertex of $G$. Denote by $N(v)$ the set vertices adjacent to $v$.
  For any disjoint union $N(v)=H\cup K$, let $G'$ be the graph obtained from $G\setminus v$ by including three new vertices $v'$, $v''$ and $x$ so that $N(v')=H\cup \{x\}$, $N(v'')=K\cup\{x\}$, and $N(x)=\{v',v''\}$ (see Figure \ref{vertexsplitting}). Then $\M(G)=\M(G')$.
\end{lem}

\begin{lem}[Star Lemma]\label{star}
Let $G$ be a weighted graph, and let $v$ be a vertex of~$G$. Let $G'$ be the graph obtained from $G$ by multiplying the weights of all edges incident to $v$ by $t>0$. Then $\M(G')=t\M(G)$.
\end{lem}

The following result is a generalization (due to Propp) of the ``urban renewal" trick first observed by Kuperberg.
\begin{figure}\centering
\begin{picture}(0,0)%
\includegraphics{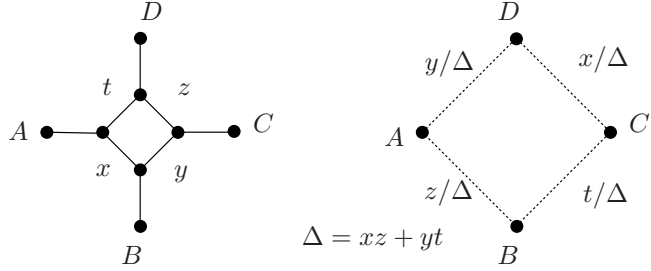}%
\end{picture}%
\setlength{\unitlength}{3947sp}%
\begingroup\makeatletter\ifx\SetFigFont\undefined%
\gdef\SetFigFont#1#2#3#4#5{%
  \reset@font\fontsize{#1}{#2pt}%
  \fontfamily{#3}\fontseries{#4}\fontshape{#5}%
  \selectfont}%
\fi\endgroup%
\begin{picture}(4054,1735)(340,-948)
\put(355,-156){\makebox(0,0)[lb]{\smash{{\SetFigFont{10}{12.0}{\familydefault}{\mddefault}{\updefault}{$A$}%
}}}}
\put(1064,-933){\makebox(0,0)[lb]{\smash{{\SetFigFont{10}{12.0}{\familydefault}{\mddefault}{\updefault}{$B$}%
}}}}
\put(1891,-106){\makebox(0,0)[lb]{\smash{{\SetFigFont{10}{12.0}{\familydefault}{\mddefault}{\updefault}{$C$}%
}}}}
\put(1182,603){\makebox(0,0)[lb]{\smash{{\SetFigFont{10}{12.0}{\familydefault}{\mddefault}{\updefault}{$D$}%
}}}}
\put(2717,-189){\makebox(0,0)[lb]{\smash{{\SetFigFont{10}{12.0}{\familydefault}{\mddefault}{\updefault}{$A$}%
}}}}
\put(3426,-933){\makebox(0,0)[lb]{\smash{{\SetFigFont{10}{12.0}{\familydefault}{\mddefault}{\updefault}{$B$}%
}}}}
\put(4253,-106){\makebox(0,0)[lb]{\smash{{\SetFigFont{10}{12.0}{\familydefault}{\mddefault}{\updefault}{$C$}%
}}}}
\put(3426,603){\makebox(0,0)[lb]{\smash{{\SetFigFont{10}{12.0}{\familydefault}{\mddefault}{\updefault}{$D$}%
}}}}
\put(904,-382){\makebox(0,0)[lb]{\smash{{\SetFigFont{10}{12.0}{\familydefault}{\mddefault}{\updefault}{$x$}%
}}}}
\put(1396,-388){\makebox(0,0)[lb]{\smash{{\SetFigFont{10}{12.0}{\familydefault}{\mddefault}{\updefault}{$y$}%
}}}}
\put(1418,130){\makebox(0,0)[lb]{\smash{{\SetFigFont{10}{12.0}{\familydefault}{\mddefault}{\updefault}{$z$}%
}}}}
\put(946,130){\makebox(0,0)[lb]{\smash{{\SetFigFont{10}{12.0}{\familydefault}{\mddefault}{\updefault}{$t$}%
}}}}
\put(2968,284){\makebox(0,0)[lb]{\smash{{\SetFigFont{10}{12.0}{\familydefault}{\mddefault}{\updefault}{$y/\Delta$}%
}}}}
\put(3934,311){\makebox(0,0)[lb]{\smash{{\SetFigFont{10}{12.0}{\familydefault}{\mddefault}{\updefault}{$x/\Delta$}%
}}}}
\put(3964,-544){\makebox(0,0)[lb]{\smash{{\SetFigFont{10}{12.0}{\familydefault}{\mddefault}{\updefault}{$t/\Delta$}%
}}}}
\put(2965,-526){\makebox(0,0)[lb]{\smash{{\SetFigFont{10}{12.0}{\familydefault}{\mddefault}{\updefault}{$z/\Delta$}%
}}}}
\put(2197,-817){\makebox(0,0)[lb]{\smash{{\SetFigFont{10}{12.0}{\familydefault}{\mddefault}{\updefault}{$\Delta= xz+yt$}%
}}}}
\end{picture}%
\caption{Urban renewal.}
\label{spider1}
\end{figure}

\begin{lem} [Spider Lemma]\label{spider}
 Let $G$ be a weighted graph containing the subgraph $K$ shown on the left in Figure \ref{spider1} (the labels indicate weights, unlabeled edges have weight 1). Suppose in addition that the four inner black vertices in the subgraph $K$, different from $A,B,C,D$, have no neighbors outside $K$. Let $G'$ be the graph obtained from $G$ by replacing $K$ by the graph $\overline{K}$ shown on right in Figure \ref{spider1}, where the dashed lines indicate new edges, weighted as shown. Then $\M(G)=(xz+yt)\M(G')$.
\end{lem}

A \textit{forced edge} of a graph $G$ is an edge contained in every perfect matching of $G$. Assume that $G$ is a weighted graph with weight assignment $\wt$ on its edges, and $G'$ is obtained from $G$ by removing forced edges $e_1,\dotsc,e_k$, and removing the vertices incident to those edges. Then one clearly has
\begin{equation*}
\M(G)=\M(G')\prod_{i=1}^k\wt(e_i).
\end{equation*}
\textit{Hereafter, whenever we remove some forced edges, we remove also the vertices incident to them.}

\medskip

Denote by $AR_{m,n}$ the dual graph of the Aztec rectangle $\mathcal{AR}_{m,n}$ rotated $45^0$ clockwise. The graph $AR_{m,n}$ consists of $m$ rows and $n$ columns of $4$-cycles (see shaded diamonds in Figure \ref{Aztecgraph}(a)). We call this graph an \emph{Aztec rectangle graph}. If one removes all bottommost vertices of $AR_{m,n}$, the resulting graph is denoted by $AR_{m-\frac12,n}$, and called a \textit{baseless Aztec rectangle graph} (see Figure \ref{Aztecgraph}(b) for an example).

\begin{figure}\center
\includegraphics[width=10cm]{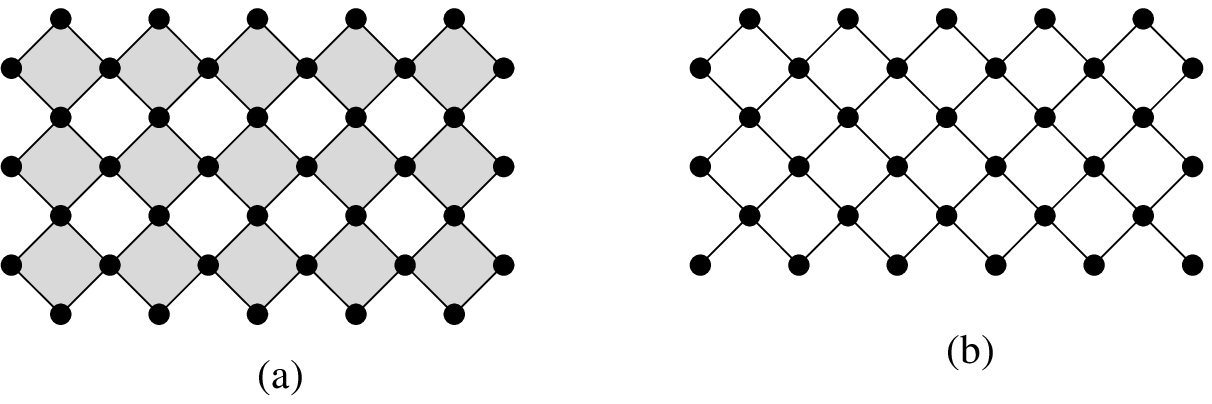}
\caption{(a) The Aztec rectangle graph $AR_{3,5}$ with diamond-faces shaded. (b) The baseless Aztec rectangle graph $AR_{3-\frac{1}{2},5}$.}
\label{Aztecgraph}
\end{figure}

Assume $a,b,c,d,q$ are positive numbers.  We consider the weight assignment $\wt_{c,d}^{a,b}(q)$ on the edges of $AR_{m,n}$ as follows. The diamond-face on row $i$ (from bottom to top) and column $j$ (from left to right) have edge-weights $a,b,dq^{i+j-2},cq^{i+j-2}$ (in clockwise order, starting from the northwest edge). See the left picture in Figure \ref{T1b} for the case $m=3$ and $n=4$. Denote by $AR_{m,n}\left(\wt_{c,d}^{a,b}(q)\right)$ the resulting weighted Aztec rectangle graph; and, similar to the unweighted case, denote by $AR_{m-\frac12,n}\left(\wt_{c,d}^{a,b}(q)\right)$ the weighted baseless Aztec rectangle graph obtained from $AR_{m,n}\left(\wt_{c,d}^{a,b}(q)\right)$ by removing the bottommost vertices.

The \textit{connected sum} $G\#G'$ of two disjoint graphs $G$ and $G'$ along the ordered sets of vertices $\{v_1,\dotsc,v_n\}\subset V(G)$ and $\{v'_1,\dotsc,v'_n\}\subset V(G')$ is the graph obtained from $G$ and $G'$ by identifying vertices $v_i$ and $v'_i$, for $i=1,\dotsc,n$.

\medskip

\begin{figure}\centering
\resizebox{!}{5cm}{
\begin{picture}(0,0)%
\includegraphics{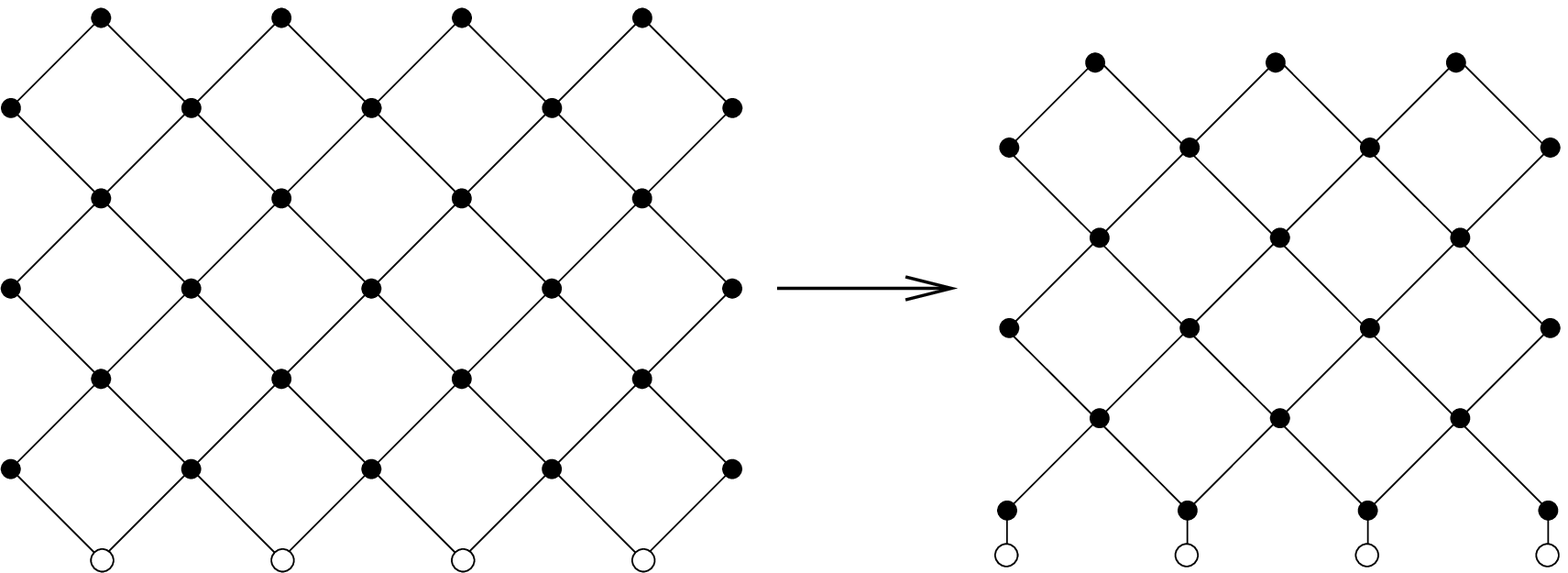}%
\end{picture}%
\setlength{\unitlength}{3947sp}%
\begingroup\makeatletter\ifx\SetFigFont\undefined%
\gdef\SetFigFont#1#2#3#4#5{%
  \reset@font\fontsize{#1}{#2pt}%
  \fontfamily{#3}\fontseries{#4}\fontshape{#5}%
  \selectfont}%
\fi\endgroup%
\begin{picture}(8177,2970)(648,-3242)
\put(888,-3172){\makebox(0,0)[lb]{\smash{{\SetFigFont{10}{12.0}{\rmdefault}{\mddefault}{\updefault}{$c$}%
}}}}
\put(1360,-3172){\makebox(0,0)[lb]{\smash{{\SetFigFont{10}{12.0}{\rmdefault}{\mddefault}{\updefault}{$d$}%
}}}}
\put(753,-2505){\makebox(0,0)[lb]{\smash{{\SetFigFont{10}{12.0}{\rmdefault}{\mddefault}{\updefault}{$a$}%
}}}}
\put(1495,-2490){\makebox(0,0)[lb]{\smash{{\SetFigFont{10}{12.0}{\rmdefault}{\mddefault}{\updefault}{$b$}%
}}}}
\put(1713,-2490){\makebox(0,0)[lb]{\smash{{\SetFigFont{10}{12.0}{\rmdefault}{\mddefault}{\updefault}{$a$}%
}}}}
\put(2433,-2490){\makebox(0,0)[lb]{\smash{{\SetFigFont{10}{12.0}{\rmdefault}{\mddefault}{\updefault}{$b$}%
}}}}
\put(2695,-2490){\makebox(0,0)[lb]{\smash{{\SetFigFont{10}{12.0}{\rmdefault}{\mddefault}{\updefault}{$a$}%
}}}}
\put(3370,-2490){\makebox(0,0)[lb]{\smash{{\SetFigFont{10}{12.0}{\rmdefault}{\mddefault}{\updefault}{$b$}%
}}}}
\put(3625,-2482){\makebox(0,0)[lb]{\smash{{\SetFigFont{10}{12.0}{\rmdefault}{\mddefault}{\updefault}{$a$}%
}}}}
\put(4345,-2505){\makebox(0,0)[lb]{\smash{{\SetFigFont{10}{12.0}{\rmdefault}{\mddefault}{\updefault}{$b$}%
}}}}
\put(1825,-3172){\makebox(0,0)[lb]{\smash{{\SetFigFont{10}{12.0}{\rmdefault}{\mddefault}{\updefault}{$cq$}%
}}}}
\put(2268,-3157){\makebox(0,0)[lb]{\smash{{\SetFigFont{10}{12.0}{\rmdefault}{\mddefault}{\updefault}{$dq$}%
}}}}
\put(2680,-3165){\makebox(0,0)[lb]{\smash{{\SetFigFont{10}{12.0}{\rmdefault}{\mddefault}{\updefault}{$cq^2$}%
}}}}
\put(3243,-3142){\makebox(0,0)[lb]{\smash{{\SetFigFont{10}{12.0}{\rmdefault}{\mddefault}{\updefault}{$dq^2$}%
}}}}
\put(3618,-3150){\makebox(0,0)[lb]{\smash{{\SetFigFont{10}{12.0}{\rmdefault}{\mddefault}{\updefault}{$cq^3$}%
}}}}
\put(4225,-3112){\makebox(0,0)[lb]{\smash{{\SetFigFont{10}{12.0}{\rmdefault}{\mddefault}{\updefault}{$dq^3$}%
}}}}
\put(798,-1530){\makebox(0,0)[lb]{\smash{{\SetFigFont{10}{12.0}{\rmdefault}{\mddefault}{\updefault}{$a$}%
}}}}
\put(1465,-1530){\makebox(0,0)[lb]{\smash{{\SetFigFont{10}{12.0}{\rmdefault}{\mddefault}{\updefault}{$b$}%
}}}}
\put(1758,-1515){\makebox(0,0)[lb]{\smash{{\SetFigFont{10}{12.0}{\rmdefault}{\mddefault}{\updefault}{$a$}%
}}}}
\put(2403,-1545){\makebox(0,0)[lb]{\smash{{\SetFigFont{10}{12.0}{\rmdefault}{\mddefault}{\updefault}{$b$}%
}}}}
\put(2665,-1537){\makebox(0,0)[lb]{\smash{{\SetFigFont{10}{12.0}{\rmdefault}{\mddefault}{\updefault}{$a$}%
}}}}
\put(3355,-1537){\makebox(0,0)[lb]{\smash{{\SetFigFont{10}{12.0}{\rmdefault}{\mddefault}{\updefault}{$b$}%
}}}}
\put(3603,-1545){\makebox(0,0)[lb]{\smash{{\SetFigFont{10}{12.0}{\rmdefault}{\mddefault}{\updefault}{$a$}%
}}}}
\put(4323,-1560){\makebox(0,0)[lb]{\smash{{\SetFigFont{10}{12.0}{\rmdefault}{\mddefault}{\updefault}{$b$}%
}}}}
\put(850,-442){\makebox(0,0)[lb]{\smash{{\SetFigFont{10}{12.0}{\rmdefault}{\mddefault}{\updefault}{$a$}%
}}}}
\put(1398,-457){\makebox(0,0)[lb]{\smash{{\SetFigFont{10}{12.0}{\rmdefault}{\mddefault}{\updefault}{$b$}%
}}}}
\put(1825,-442){\makebox(0,0)[lb]{\smash{{\SetFigFont{10}{12.0}{\rmdefault}{\mddefault}{\updefault}{$a$}%
}}}}
\put(2313,-480){\makebox(0,0)[lb]{\smash{{\SetFigFont{10}{12.0}{\rmdefault}{\mddefault}{\updefault}{$b$}%
}}}}
\put(2770,-442){\makebox(0,0)[lb]{\smash{{\SetFigFont{10}{12.0}{\rmdefault}{\mddefault}{\updefault}{$a$}%
}}}}
\put(3288,-480){\makebox(0,0)[lb]{\smash{{\SetFigFont{10}{12.0}{\rmdefault}{\mddefault}{\updefault}{$b$}%
}}}}
\put(3693,-457){\makebox(0,0)[lb]{\smash{{\SetFigFont{10}{12.0}{\rmdefault}{\mddefault}{\updefault}{$a$}%
}}}}
\put(4255,-502){\makebox(0,0)[lb]{\smash{{\SetFigFont{10}{12.0}{\rmdefault}{\mddefault}{\updefault}{$b$}%
}}}}
\put(820,-2160){\makebox(0,0)[lb]{\smash{{\SetFigFont{10}{12.0}{\rmdefault}{\mddefault}{\updefault}{$cq$}%
}}}}
\put(1315,-2190){\makebox(0,0)[lb]{\smash{{\SetFigFont{10}{12.0}{\rmdefault}{\mddefault}{\updefault}{$dq$}%
}}}}
\put(1720,-2190){\makebox(0,0)[lb]{\smash{{\SetFigFont{10}{12.0}{\rmdefault}{\mddefault}{\updefault}{$cq^2$}%
}}}}
\put(2268,-2175){\makebox(0,0)[lb]{\smash{{\SetFigFont{10}{12.0}{\rmdefault}{\mddefault}{\updefault}{$dq^2$}%
}}}}
\put(2650,-2175){\makebox(0,0)[lb]{\smash{{\SetFigFont{10}{12.0}{\rmdefault}{\mddefault}{\updefault}{$cq^3$}%
}}}}
\put(3220,-2182){\makebox(0,0)[lb]{\smash{{\SetFigFont{10}{12.0}{\rmdefault}{\mddefault}{\updefault}{$dq^3$}%
}}}}
\put(678,-1207){\makebox(0,0)[lb]{\smash{{\SetFigFont{10}{12.0}{\rmdefault}{\mddefault}{\updefault}{$cq^2$}%
}}}}
\put(1368,-1230){\makebox(0,0)[lb]{\smash{{\SetFigFont{10}{12.0}{\rmdefault}{\mddefault}{\updefault}{$dq^2$}%
}}}}
\put(1720,-1222){\makebox(0,0)[lb]{\smash{{\SetFigFont{10}{12.0}{\rmdefault}{\mddefault}{\updefault}{$cq^3$}%
}}}}
\put(2305,-1222){\makebox(0,0)[lb]{\smash{{\SetFigFont{10}{12.0}{\rmdefault}{\mddefault}{\updefault}{$dq^3$}%
}}}}
\put(2658,-1237){\makebox(0,0)[lb]{\smash{{\SetFigFont{10}{12.0}{\rmdefault}{\mddefault}{\updefault}{$cq^4$}%
}}}}
\put(3235,-1237){\makebox(0,0)[lb]{\smash{{\SetFigFont{10}{12.0}{\rmdefault}{\mddefault}{\updefault}{$dq^4$}%
}}}}
\put(3610,-1252){\makebox(0,0)[lb]{\smash{{\SetFigFont{10}{12.0}{\rmdefault}{\mddefault}{\updefault}{$cq^5$}%
}}}}
\put(4210,-1245){\makebox(0,0)[lb]{\smash{{\SetFigFont{10}{12.0}{\rmdefault}{\mddefault}{\updefault}{$dq^5$}%
}}}}
\put(3595,-2190){\makebox(0,0)[lb]{\smash{{\SetFigFont{10}{12.0}{\rmdefault}{\mddefault}{\updefault}{$cq^4$}%
}}}}
\put(4180,-2182){\makebox(0,0)[lb]{\smash{{\SetFigFont{10}{12.0}{\rmdefault}{\mddefault}{\updefault}{$dq^4$}%
}}}}
\put(5871,-2726){\makebox(0,0)[lb]{\smash{{\SetFigFont{10}{12.0}{\rmdefault}{\mddefault}{\updefault}{$aq$}%
}}}}
\put(6711,-2719){\makebox(0,0)[lb]{\smash{{\SetFigFont{10}{12.0}{\rmdefault}{\mddefault}{\updefault}{$b$}%
}}}}
\put(6898,-2726){\makebox(0,0)[lb]{\smash{{\SetFigFont{10}{12.0}{\rmdefault}{\mddefault}{\updefault}{$aq$}%
}}}}
\put(7641,-2726){\makebox(0,0)[lb]{\smash{{\SetFigFont{10}{12.0}{\rmdefault}{\mddefault}{\updefault}{$b$}%
}}}}
\put(7836,-2719){\makebox(0,0)[lb]{\smash{{\SetFigFont{10}{12.0}{\rmdefault}{\mddefault}{\updefault}{$aq$}%
}}}}
\put(8608,-2711){\makebox(0,0)[lb]{\smash{{\SetFigFont{10}{12.0}{\rmdefault}{\mddefault}{\updefault}{$b$}%
}}}}
\put(5991,-1736){\makebox(0,0)[lb]{\smash{{\SetFigFont{10}{12.0}{\rmdefault}{\mddefault}{\updefault}{$aq$}%
}}}}
\put(6688,-1759){\makebox(0,0)[lb]{\smash{{\SetFigFont{10}{12.0}{\rmdefault}{\mddefault}{\updefault}{$b$}%
}}}}
\put(6898,-1744){\makebox(0,0)[lb]{\smash{{\SetFigFont{10}{12.0}{\rmdefault}{\mddefault}{\updefault}{$aq$}%
}}}}
\put(7671,-1774){\makebox(0,0)[lb]{\smash{{\SetFigFont{10}{12.0}{\rmdefault}{\mddefault}{\updefault}{$b$}%
}}}}
\put(7851,-1766){\makebox(0,0)[lb]{\smash{{\SetFigFont{10}{12.0}{\rmdefault}{\mddefault}{\updefault}{$aq$}%
}}}}
\put(8631,-1781){\makebox(0,0)[lb]{\smash{{\SetFigFont{10}{12.0}{\rmdefault}{\mddefault}{\updefault}{$b$}%
}}}}
\put(5961,-2359){\makebox(0,0)[lb]{\smash{{\SetFigFont{10}{12.0}{\rmdefault}{\mddefault}{\updefault}{$cq$}%
}}}}
\put(6553,-2381){\makebox(0,0)[lb]{\smash{{\SetFigFont{10}{12.0}{\rmdefault}{\mddefault}{\updefault}{$dq$}%
}}}}
\put(6921,-2396){\makebox(0,0)[lb]{\smash{{\SetFigFont{10}{12.0}{\rmdefault}{\mddefault}{\updefault}{$cq^2$}%
}}}}
\put(7483,-2389){\makebox(0,0)[lb]{\smash{{\SetFigFont{10}{12.0}{\rmdefault}{\mddefault}{\updefault}{$dq^2$}%
}}}}
\put(7873,-2396){\makebox(0,0)[lb]{\smash{{\SetFigFont{10}{12.0}{\rmdefault}{\mddefault}{\updefault}{$cq^3$}%
}}}}
\put(8428,-2404){\makebox(0,0)[lb]{\smash{{\SetFigFont{10}{12.0}{\rmdefault}{\mddefault}{\updefault}{$dq^3$}%
}}}}
\put(5976,-784){\makebox(0,0)[lb]{\smash{{\SetFigFont{10}{12.0}{\rmdefault}{\mddefault}{\updefault}{$aq$}%
}}}}
\put(6688,-784){\makebox(0,0)[lb]{\smash{{\SetFigFont{10}{12.0}{\rmdefault}{\mddefault}{\updefault}{$b$}%
}}}}
\put(6921,-776){\makebox(0,0)[lb]{\smash{{\SetFigFont{10}{12.0}{\rmdefault}{\mddefault}{\updefault}{$aq$}%
}}}}
\put(7626,-791){\makebox(0,0)[lb]{\smash{{\SetFigFont{10}{12.0}{\rmdefault}{\mddefault}{\updefault}{$b$}%
}}}}
\put(7896,-769){\makebox(0,0)[lb]{\smash{{\SetFigFont{10}{12.0}{\rmdefault}{\mddefault}{\updefault}{$aq$}%
}}}}
\put(8586,-784){\makebox(0,0)[lb]{\smash{{\SetFigFont{10}{12.0}{\rmdefault}{\mddefault}{\updefault}{$b$}%
}}}}
\put(5923,-1406){\makebox(0,0)[lb]{\smash{{\SetFigFont{10}{12.0}{\rmdefault}{\mddefault}{\updefault}{$cq^2$}%
}}}}
\put(6568,-1429){\makebox(0,0)[lb]{\smash{{\SetFigFont{10}{12.0}{\rmdefault}{\mddefault}{\updefault}{$dq^2$}%
}}}}
\put(6906,-1436){\makebox(0,0)[lb]{\smash{{\SetFigFont{10}{12.0}{\rmdefault}{\mddefault}{\updefault}{$cq^3$}%
}}}}
\put(7506,-1429){\makebox(0,0)[lb]{\smash{{\SetFigFont{10}{12.0}{\rmdefault}{\mddefault}{\updefault}{$dq^3$}%
}}}}
\put(7858,-1444){\makebox(0,0)[lb]{\smash{{\SetFigFont{10}{12.0}{\rmdefault}{\mddefault}{\updefault}{$cq^4$}%
}}}}
\put(8443,-1429){\makebox(0,0)[lb]{\smash{{\SetFigFont{10}{12.0}{\rmdefault}{\mddefault}{\updefault}{$dq^4$}%
}}}}
\end{picture}}
\caption{Illustrating the replacement rule in  Lemma \ref{T1}. The white circles indicate the vertices $v_1,v_2,\dotsc,v_n$.}
\label{T1b}
\end{figure}

\begin{lem}\label{T1}
Let $G$ be a graph and let  $\{v_1,\dotsc,v_n\}$ be an ordered subset of its vertices. Then

\begin{equation}\label{T1eq1}
\M\left(AR_{m,n}\left(\wt_{c,d}^{a,b}(q)\right)\#G\right)=(ad+bc)^{m}q^{\frac{m(m-1)}{2}}\M\left({}_|AR_{m-\frac12,n-1}\left(\wt_{c,d}^{aq,b}(q)\right)\#G\right),
\end{equation}
where  ${}_|AR_{m-\frac12,n-1}\left(\wt_{c,d}^{a,b}(q)\right)$ is obtained from the graph $AR_{m-\frac12,n-1}\left(\wt_{c,d}^{a,b}(q)\right)$ by appending vertical edges from their bottommost vertices;  and where the connected sum acts on $G$ along $\{v_1,\dotsc,v_n\}$, and on other summands along their bottommost vertices (ordered from left to right).
\end{lem}

The replacement in Lemma \ref{T1} is illustrated in Figure \ref{T1b} for $m=3$ and $n=4$. We note that the unweighted version (when $a=b=c=d=q=1$) of Lemma \ref{T1}  was introduced in \cite{Tri} (see Lemma 3.5).

\begin{figure}\centering
\resizebox{!}{12cm}{
\begin{picture}(0,0)%
\includegraphics{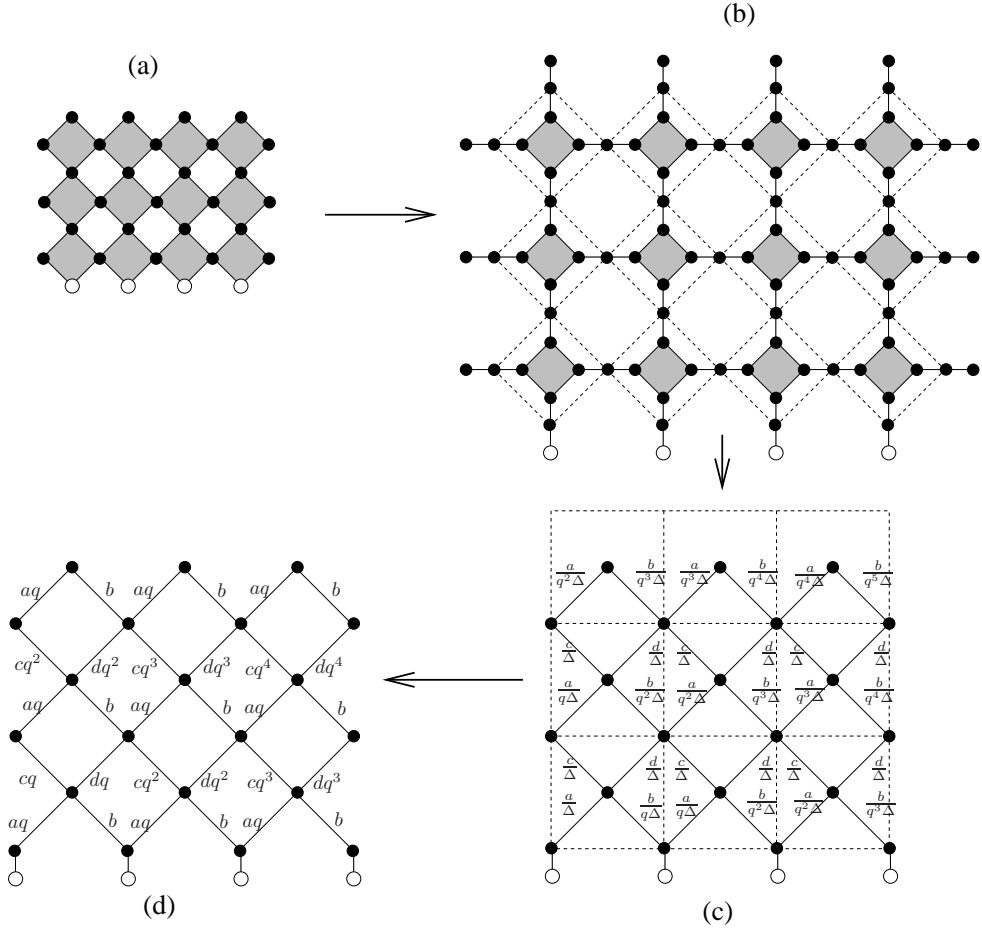}%
\end{picture}%
\setlength{\unitlength}{3947sp}%
\begingroup\makeatletter\ifx\SetFigFont\undefined%
\gdef\SetFigFont#1#2#3#4#5{%
  \reset@font\fontsize{#1}{#2pt}%
  \fontfamily{#3}\fontseries{#4}\fontshape{#5}%
  \selectfont}%
\fi\endgroup%
\begin{picture}(8167,7567)(1104,-10469)
\put(5694,-7516){\makebox(0,0)[lb]{\smash{{\SetFigFont{10}{12.0}{\rmdefault}{\mddefault}{\updefault}{$\frac{a}{q^2\Delta}$}%
}}}}
\put(6369,-7509){\makebox(0,0)[lb]{\smash{{\SetFigFont{10}{12.0}{\rmdefault}{\mddefault}{\updefault}{$\frac{b}{q^3\Delta}$}%
}}}}
\put(6736,-7509){\makebox(0,0)[lb]{\smash{{\SetFigFont{10}{12.0}{\rmdefault}{\mddefault}{\updefault}{$\frac{a}{q^3\Delta}$}%
}}}}
\put(7306,-7509){\makebox(0,0)[lb]{\smash{{\SetFigFont{10}{12.0}{\rmdefault}{\mddefault}{\updefault}{$\frac{b}{q^4\Delta}$}%
}}}}
\put(7704,-7531){\makebox(0,0)[lb]{\smash{{\SetFigFont{10}{12.0}{\rmdefault}{\mddefault}{\updefault}{$\frac{a}{q^4\Delta}$}%
}}}}
\put(8274,-7516){\makebox(0,0)[lb]{\smash{{\SetFigFont{10}{12.0}{\rmdefault}{\mddefault}{\updefault}{$\frac{b}{q^5\Delta}$}%
}}}}
\put(5746,-8169){\makebox(0,0)[lb]{\smash{{\SetFigFont{10}{12.0}{\rmdefault}{\mddefault}{\updefault}{$\frac{c}{\Delta}$}%
}}}}
\put(6476,-8191){\makebox(0,0)[lb]{\smash{{\SetFigFont{10}{12.0}{\rmdefault}{\mddefault}{\updefault}{$\frac{d}{\Delta}$}%
}}}}
\put(6712,-8191){\makebox(0,0)[lb]{\smash{{\SetFigFont{10}{12.0}{\rmdefault}{\mddefault}{\updefault}{$\frac{c}{\Delta}$}%
}}}}
\put(7421,-8191){\makebox(0,0)[lb]{\smash{{\SetFigFont{10}{12.0}{\rmdefault}{\mddefault}{\updefault}{$\frac{d}{\Delta}$}%
}}}}
\put(7657,-8191){\makebox(0,0)[lb]{\smash{{\SetFigFont{10}{12.0}{\rmdefault}{\mddefault}{\updefault}{$\frac{c}{\Delta}$}%
}}}}
\put(8366,-8191){\makebox(0,0)[lb]{\smash{{\SetFigFont{10}{12.0}{\rmdefault}{\mddefault}{\updefault}{$\frac{d}{\Delta}$}%
}}}}
\put(5709,-8499){\makebox(0,0)[lb]{\smash{{\SetFigFont{10}{12.0}{\rmdefault}{\mddefault}{\updefault}{$\frac{a}{q\Delta}$}%
}}}}
\put(6361,-8499){\makebox(0,0)[lb]{\smash{{\SetFigFont{10}{12.0}{\rmdefault}{\mddefault}{\updefault}{$\frac{b}{q^2\Delta}$}%
}}}}
\put(6714,-8506){\makebox(0,0)[lb]{\smash{{\SetFigFont{10}{12.0}{\rmdefault}{\mddefault}{\updefault}{$\frac{a}{q^2\Delta}$}%
}}}}
\put(7336,-8499){\makebox(0,0)[lb]{\smash{{\SetFigFont{10}{12.0}{\rmdefault}{\mddefault}{\updefault}{$\frac{b}{q^3\Delta}$}%
}}}}
\put(7704,-8491){\makebox(0,0)[lb]{\smash{{\SetFigFont{10}{12.0}{\rmdefault}{\mddefault}{\updefault}{$\frac{a}{q^3\Delta}$}%
}}}}
\put(8281,-8499){\makebox(0,0)[lb]{\smash{{\SetFigFont{10}{12.0}{\rmdefault}{\mddefault}{\updefault}{$\frac{b}{q^4\Delta}$}%
}}}}
\put(5754,-9144){\makebox(0,0)[lb]{\smash{{\SetFigFont{10}{12.0}{\rmdefault}{\mddefault}{\updefault}{$\frac{c}{\Delta}$}%
}}}}
\put(6453,-9158){\makebox(0,0)[lb]{\smash{{\SetFigFont{10}{12.0}{\rmdefault}{\mddefault}{\updefault}{$\frac{d}{\Delta}$}%
}}}}
\put(6689,-9158){\makebox(0,0)[lb]{\smash{{\SetFigFont{10}{12.0}{\rmdefault}{\mddefault}{\updefault}{$\frac{c}{\Delta}$}%
}}}}
\put(7398,-9158){\makebox(0,0)[lb]{\smash{{\SetFigFont{10}{12.0}{\rmdefault}{\mddefault}{\updefault}{$\frac{d}{\Delta}$}%
}}}}
\put(7634,-9158){\makebox(0,0)[lb]{\smash{{\SetFigFont{10}{12.0}{\rmdefault}{\mddefault}{\updefault}{$\frac{c}{\Delta}$}%
}}}}
\put(8343,-9158){\makebox(0,0)[lb]{\smash{{\SetFigFont{10}{12.0}{\rmdefault}{\mddefault}{\updefault}{$\frac{d}{\Delta}$}%
}}}}
\put(5746,-9444){\makebox(0,0)[lb]{\smash{{\SetFigFont{10}{12.0}{\rmdefault}{\mddefault}{\updefault}{$\frac{a}{\Delta}$}%
}}}}
\put(6399,-9474){\makebox(0,0)[lb]{\smash{{\SetFigFont{10}{12.0}{\rmdefault}{\mddefault}{\updefault}{$\frac{b}{q\Delta}$}%
}}}}
\put(6699,-9459){\makebox(0,0)[lb]{\smash{{\SetFigFont{10}{12.0}{\rmdefault}{\mddefault}{\updefault}{$\frac{a}{q\Delta}$}%
}}}}
\put(7299,-9436){\makebox(0,0)[lb]{\smash{{\SetFigFont{10}{12.0}{\rmdefault}{\mddefault}{\updefault}{$\frac{b}{q^2\Delta}$}%
}}}}
\put(7689,-9429){\makebox(0,0)[lb]{\smash{{\SetFigFont{10}{12.0}{\rmdefault}{\mddefault}{\updefault}{$\frac{a}{q^2\Delta}$}%
}}}}
\put(8296,-9451){\makebox(0,0)[lb]{\smash{{\SetFigFont{10}{12.0}{\rmdefault}{\mddefault}{\updefault}{$\frac{b}{q^3\Delta}$}%
}}}}
\put(1119,-9611){\makebox(0,0)[lb]{\smash{{\SetFigFont{10}{12.0}{\rmdefault}{\mddefault}{\updefault}{$aq$}%
}}}}
\put(1959,-9604){\makebox(0,0)[lb]{\smash{{\SetFigFont{10}{12.0}{\rmdefault}{\mddefault}{\updefault}{$b$}%
}}}}
\put(2146,-9611){\makebox(0,0)[lb]{\smash{{\SetFigFont{10}{12.0}{\rmdefault}{\mddefault}{\updefault}{$aq$}%
}}}}
\put(2889,-9611){\makebox(0,0)[lb]{\smash{{\SetFigFont{10}{12.0}{\rmdefault}{\mddefault}{\updefault}{$b$}%
}}}}
\put(3084,-9604){\makebox(0,0)[lb]{\smash{{\SetFigFont{10}{12.0}{\rmdefault}{\mddefault}{\updefault}{$aq$}%
}}}}
\put(3856,-9596){\makebox(0,0)[lb]{\smash{{\SetFigFont{10}{12.0}{\rmdefault}{\mddefault}{\updefault}{$b$}%
}}}}
\put(1239,-8621){\makebox(0,0)[lb]{\smash{{\SetFigFont{10}{12.0}{\rmdefault}{\mddefault}{\updefault}{$aq$}%
}}}}
\put(1936,-8644){\makebox(0,0)[lb]{\smash{{\SetFigFont{10}{12.0}{\rmdefault}{\mddefault}{\updefault}{$b$}%
}}}}
\put(2146,-8629){\makebox(0,0)[lb]{\smash{{\SetFigFont{10}{12.0}{\rmdefault}{\mddefault}{\updefault}{$aq$}%
}}}}
\put(2919,-8659){\makebox(0,0)[lb]{\smash{{\SetFigFont{10}{12.0}{\rmdefault}{\mddefault}{\updefault}{$b$}%
}}}}
\put(3099,-8651){\makebox(0,0)[lb]{\smash{{\SetFigFont{10}{12.0}{\rmdefault}{\mddefault}{\updefault}{$aq$}%
}}}}
\put(3879,-8666){\makebox(0,0)[lb]{\smash{{\SetFigFont{10}{12.0}{\rmdefault}{\mddefault}{\updefault}{$b$}%
}}}}
\put(1209,-9244){\makebox(0,0)[lb]{\smash{{\SetFigFont{10}{12.0}{\rmdefault}{\mddefault}{\updefault}{$cq$}%
}}}}
\put(1801,-9266){\makebox(0,0)[lb]{\smash{{\SetFigFont{10}{12.0}{\rmdefault}{\mddefault}{\updefault}{$dq$}%
}}}}
\put(2169,-9281){\makebox(0,0)[lb]{\smash{{\SetFigFont{10}{12.0}{\rmdefault}{\mddefault}{\updefault}{$cq^2$}%
}}}}
\put(2731,-9274){\makebox(0,0)[lb]{\smash{{\SetFigFont{10}{12.0}{\rmdefault}{\mddefault}{\updefault}{$dq^2$}%
}}}}
\put(3121,-9281){\makebox(0,0)[lb]{\smash{{\SetFigFont{10}{12.0}{\rmdefault}{\mddefault}{\updefault}{$cq^3$}%
}}}}
\put(3676,-9289){\makebox(0,0)[lb]{\smash{{\SetFigFont{10}{12.0}{\rmdefault}{\mddefault}{\updefault}{$dq^3$}%
}}}}
\put(1224,-7669){\makebox(0,0)[lb]{\smash{{\SetFigFont{10}{12.0}{\rmdefault}{\mddefault}{\updefault}{$aq$}%
}}}}
\put(1936,-7669){\makebox(0,0)[lb]{\smash{{\SetFigFont{10}{12.0}{\rmdefault}{\mddefault}{\updefault}{$b$}%
}}}}
\put(2169,-7661){\makebox(0,0)[lb]{\smash{{\SetFigFont{10}{12.0}{\rmdefault}{\mddefault}{\updefault}{$aq$}%
}}}}
\put(2874,-7676){\makebox(0,0)[lb]{\smash{{\SetFigFont{10}{12.0}{\rmdefault}{\mddefault}{\updefault}{$b$}%
}}}}
\put(3144,-7654){\makebox(0,0)[lb]{\smash{{\SetFigFont{10}{12.0}{\rmdefault}{\mddefault}{\updefault}{$aq$}%
}}}}
\put(3834,-7669){\makebox(0,0)[lb]{\smash{{\SetFigFont{10}{12.0}{\rmdefault}{\mddefault}{\updefault}{$b$}%
}}}}
\put(1171,-8291){\makebox(0,0)[lb]{\smash{{\SetFigFont{10}{12.0}{\rmdefault}{\mddefault}{\updefault}{$cq^2$}%
}}}}
\put(1816,-8314){\makebox(0,0)[lb]{\smash{{\SetFigFont{10}{12.0}{\rmdefault}{\mddefault}{\updefault}{$dq^2$}%
}}}}
\put(2154,-8321){\makebox(0,0)[lb]{\smash{{\SetFigFont{10}{12.0}{\rmdefault}{\mddefault}{\updefault}{$cq^3$}%
}}}}
\put(2754,-8314){\makebox(0,0)[lb]{\smash{{\SetFigFont{10}{12.0}{\rmdefault}{\mddefault}{\updefault}{$dq^3$}%
}}}}
\put(3106,-8329){\makebox(0,0)[lb]{\smash{{\SetFigFont{10}{12.0}{\rmdefault}{\mddefault}{\updefault}{$cq^4$}%
}}}}
\put(3691,-8314){\makebox(0,0)[lb]{\smash{{\SetFigFont{10}{12.0}{\rmdefault}{\mddefault}{\updefault}{$dq^4$}%
}}}}
\end{picture}}
\caption{Illustrating the proof of Lemma \ref{T1}. }
\label{ProofT2}
\end{figure}

\begin{proof}
The proof is based on Figure \ref{ProofT2}, for $m=3$ and $n=4$.  First, we apply Vertex-Splitting Lemma \ref{VS} to vertices of ${}_|AR_{m,n}\left(\wt_{c,d}^{a,b}(q)\right)$ in the graph on the left-hand side of (\ref{T1eq1}) as in Figures \ref{ProofT2}(a) and (b); the sides of the shaded diamonds are weighted as in the left picture in Figure \ref{T1b}. Denote by $G_1$ the resulting graph.

Next, we apply Spider Lemma \ref{spider} around all shaded diamonds in $G_1$, and remove all $m$ leftmost horizontal edges, $m$ rightmost horizontal edges and $n$ topmost vertical edges, which are forced (see Figure \ref{ProofT2}(b)). We get the graph $G_2=AR_{m-\frac{1}{2},n-1}(\wt')\#G$, where $AR_{m-\frac{1}{2},n-1}(\wt')$ is a weighted version of $\AR_{m-\frac{1}{2},n-1}$ with edges weighted as in Figure \ref{ProofT2}(c), and where $\Delta=ad+bc$.

Finally, we use Star Lemma \ref{star} to change the edge-weights in the graph $G_2$. Divide the graph $\AR_{m-\frac{1}{2},n-1}(\wt')$, except for vertical edges, into $m(n-1)$ subgraphs restricted by dotted squares in Figure \ref{ProofT2}(c). Apply the Star Lemma with factor $q^{i+j-1}\Delta$ to the central vertex of the dotted square in row $i$ and column $j$, for $i=1,2,\dotsc,m$ and $j=1,2,\dotsc,n-1$. This way, we obtain the graph on the right-hand side of (\ref{T1eq1}).

By Vertex-Splitting, Spider and Star Lemmas, we get
\begin{align}
\M&\left({}_|AR_{m,n}\left(\wt_{c,d}^{a,b}(q)\right)\#G\right)=\M(G_1)\\
&=\M(G_2)\prod_{1\leq i, j\leq n} (q^{i+j-2}\Delta)\\
&=\M\left(AR_{m-\frac{1}{2},n-1}(\wt_{c,d}^{aq,b}(q))\right) \prod_{1\leq i \leq m, 1\leq j\leq n-1} \left(q^{i+j-1}\Delta\right)^{-1} \prod_{1\leq i, j\leq n} \left(q^{i+j-2}\Delta\right),
\end{align}
which implies (\ref{T1eq1}).
\end{proof}

\section{Matching generating function of  weighted Aztec rectangle graphs}
In this section, we give an explicit formula for the matching generating function of the graph $\AR_{m,n}\left(\wt_{c,d}^{a,b}(q)\right)$, where $m-n$ vertices on the base have been removed.

Given $\lambda_1\geq \lambda_2\geq\dotsc \lambda_k\geq 0$, a \emph{plane partition} of shape $(\lambda_1,\lambda_2,\dotsc,\lambda_k)$ is an array of non-negative integers of the form 
\begin{center}
\begin{tabular}{rccccccccc}
$n_{1,1}$   &$n_{1,2}$                 &$n_{1,3}$               & $\dotsc$               &  $\dotsc$                        & $\dotsc$                            &   $n_{1,\lambda_1}$ \\\noalign{\smallskip\smallskip}
$n_{2,1}$   &  $n_{2,2}$              & $n_{2,3}$             &  $\dotsc$               & $\dotsc$                        &         $n_{2,\lambda_2}$&          \\\noalign{\smallskip\smallskip}
$\vdots$    &       $\vdots$            & $\vdots$                &        $\vdots$         &     $\vdots$                &    &              \\\noalign{\smallskip\smallskip}
 $n_{k,1}$  &  $n_{k,2}$               & $n_{k,3}$              &     $\dotsc$             &   $n_{k,\lambda_k}$ &                                          &           \\\noalign{\smallskip\smallskip}
\end{tabular},
\end{center}
where the entries are weakly decreasing across the rows and down the columns. A \emph{column-strict plane partition} is a plane partition having entries in each column strictly decreasing. We refer reader to \cite{Stanley} for properties of column-strict plane partitions.

A \emph{semihexagon} $\mathcal{SH}_{a,b}$ is the upper half of the semi-regular hexagon of side-lengths $a,b,b,a,b,b$ (in clockwise order, starting from the northwest side) on the triangular lattice. We are interested in the (lozenge) tilings of the semihexagon $\mathcal{SH}_{a,b}$ with $a$ up-pointing unit triangles removed from the base, which are called \emph{dents}. Assume that the positions of the dents are $1\leq s_1<s_2<\dotsc<s_a\leq a+b$, we denote by $\mathcal{SH}_{a,b}(s_1,\dotsc,s_a)$ the semihexagon with dents (see Figure \ref{semihexagonex} for an example; the black unit triangles indicate the dents).

\begin{figure}\centering
\includegraphics[width=5cm]{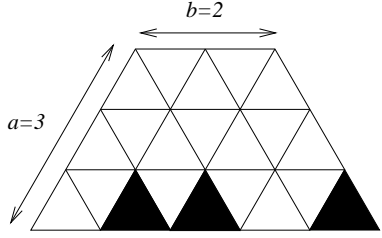}
\caption{The semihexagon with dents $\mathcal{SH}_{3,2}(2,3,5)$.}
\label{semihexagonex}
\end{figure}

\begin{thm}\label{weightrectangle}
Assume that $m$ and $n$ are two positive integers, so that $m<n$. The matching generating function of the weighted Aztec rectangle graph $AR_{m,n}\left(\wt_{c,d}^{a,b}(q)\right)$, where all bottommost vertices, except for the $s_1$-st, the $s_2$-nd, $\dotsc$ and the $s_m$-th ones, have been removed, equals 
\begin{equation}
q^{\frac{(m-1)m(m+1)}{3}+\sum_{i=1}^{m}(s_i-i)}a^{\sum_{i=1}^{m}(s_i-i)}b^{m(n-m)-\sum_{i=1}^{m}(s_i-i)}\prod_{k=1}^{m}\Delta_{k}^{m-k+1}\cdot\prod_{1\leq i<j\leq m}\frac{q^{s_j}-q^{s_i}}{q^{j}-q^{i}},
\end{equation}
where $\Delta_k=adq^{k-1}+bc$.
\end{thm}

\begin{figure}\centering
\includegraphics[width=7cm]{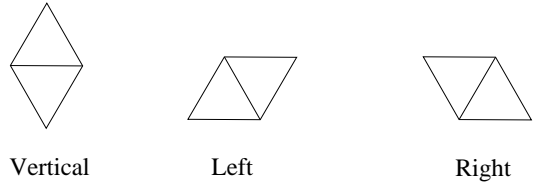}
\caption{Three types of rhombi.}
\label{rhumbustype}
\end{figure}

\begin{proof}
Denote by $G$ our Aztec rectangle graph with $n-m$ bottommost vertices removed.
Consider the graph $G'$ obtained from $AR_{m,n}\left(\wt_{c,d}^{a,b}(t,q)\right)$ by adding a vertical edge (with weight 1) at all bottommost vertices, except for the ones at the positions $s_i$'s (see Figure \ref{holeytransform2}(a), for $m=3$, $n=5$, $s_1=2$, $s_2=3$, $s_3=5$). Then by considering forced edges, we get $\M(G)=\M(G')$.

\medskip

Next, we apply a $m$-step transforming process based on Figure \ref{holeytransform2} as follows. First, apply the replacement rule in Lemma \ref{T1} as in Figures \ref{holeytransform2}(a) and (b): the part above the dotted line in graph (a) is replaced by the part above that line in graph (b). Second, we apply the same rule to replace the part above the upper dotted line in graph (b) by the part above the lower dotted line in graph (c). Keep doing this process until we eliminate all  rows of diamonds on the top of the resulting graph. Denote by $G''$ the final graph (see Figure \ref{holeytransform2}(d)).

By removing vertical forced edges at the bottom of $G''$, we get a the dual graph $\widetilde{G}$ of a weighted semi-hexagon with dents $\mathcal{SH}_{m,n-m}(s_1,\dotsc,s_{m})$. 
In particular, the \textit{left rhombi} on the level $k$ (the bottom is at the level 0) are weighted by $aq^{k+1}$, all \textit{right rhombi} are weighted by $b$, and all \textit{vertical rhombi} have weight 1 (see Figure \ref{rhumbustype} for three types of rhombi). By Lemma \ref{T1}, we obtain
\begin{equation}\label{masstransform}
\M\left(AR_{m,n}\left(\wt_{c,d}^{a,b}(q)\right)\right)=q^{\frac{(m-1)m(m+1)}{3}}\left(\prod_{k=1}^{m}\Delta_{k}^{m-k+1}\right) \M(\widetilde{G}).
\end{equation}

\begin{figure}\centering
\resizebox{!}{12cm}{
\begin{picture}(0,0)%
\includegraphics{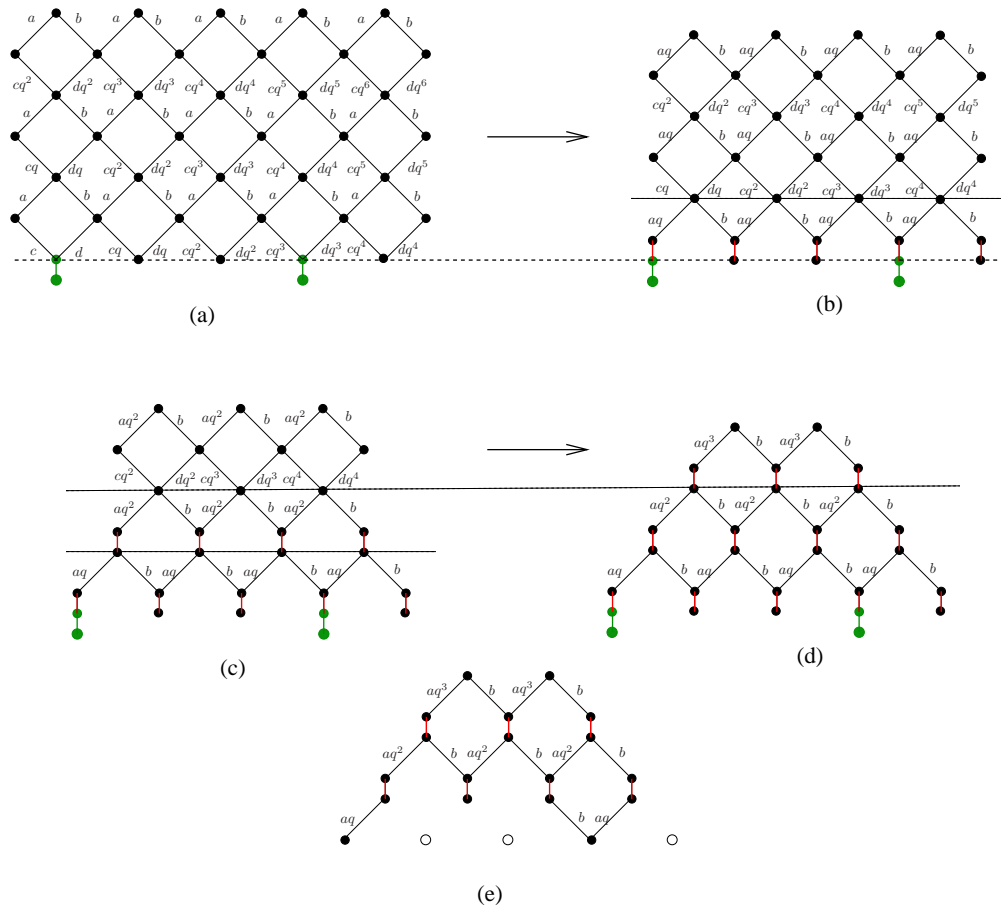}%
\end{picture}%
\setlength{\unitlength}{3947sp}%
\begingroup\makeatletter\ifx\SetFigFont\undefined%
\gdef\SetFigFont#1#2#3#4#5{%
  \reset@font\fontsize{#1}{#2pt}%
  \fontfamily{#3}\fontseries{#4}\fontshape{#5}%
  \selectfont}%
\fi\endgroup%
\begin{picture}(11413,10377)(1130,-20095)
\put(1359,-12564){\makebox(0,0)[lb]{\smash{{\SetFigFont{10}{12.0}{\rmdefault}{\mddefault}{\updefault}{$c$}%
}}}}
\put(1876,-12586){\makebox(0,0)[lb]{\smash{{\SetFigFont{10}{12.0}{\rmdefault}{\mddefault}{\updefault}{$d$}%
}}}}
\put(1235,-11951){\makebox(0,0)[lb]{\smash{{\SetFigFont{10}{12.0}{\rmdefault}{\mddefault}{\updefault}{$a$}%
}}}}
\put(1977,-11936){\makebox(0,0)[lb]{\smash{{\SetFigFont{10}{12.0}{\rmdefault}{\mddefault}{\updefault}{$b$}%
}}}}
\put(2195,-11936){\makebox(0,0)[lb]{\smash{{\SetFigFont{10}{12.0}{\rmdefault}{\mddefault}{\updefault}{$a$}%
}}}}
\put(2915,-11936){\makebox(0,0)[lb]{\smash{{\SetFigFont{10}{12.0}{\rmdefault}{\mddefault}{\updefault}{$b$}%
}}}}
\put(3177,-11936){\makebox(0,0)[lb]{\smash{{\SetFigFont{10}{12.0}{\rmdefault}{\mddefault}{\updefault}{$a$}%
}}}}
\put(3852,-11936){\makebox(0,0)[lb]{\smash{{\SetFigFont{10}{12.0}{\rmdefault}{\mddefault}{\updefault}{$b$}%
}}}}
\put(4107,-11928){\makebox(0,0)[lb]{\smash{{\SetFigFont{10}{12.0}{\rmdefault}{\mddefault}{\updefault}{$a$}%
}}}}
\put(4827,-11951){\makebox(0,0)[lb]{\smash{{\SetFigFont{10}{12.0}{\rmdefault}{\mddefault}{\updefault}{$b$}%
}}}}
\put(2259,-12556){\makebox(0,0)[lb]{\smash{{\SetFigFont{10}{12.0}{\rmdefault}{\mddefault}{\updefault}{$cq$}%
}}}}
\put(2769,-12564){\makebox(0,0)[lb]{\smash{{\SetFigFont{10}{12.0}{\rmdefault}{\mddefault}{\updefault}{$dq$}%
}}}}
\put(3106,-12564){\makebox(0,0)[lb]{\smash{{\SetFigFont{10}{12.0}{\rmdefault}{\mddefault}{\updefault}{$cq^2$}%
}}}}
\put(3725,-12588){\makebox(0,0)[lb]{\smash{{\SetFigFont{10}{12.0}{\rmdefault}{\mddefault}{\updefault}{$dq^2$}%
}}}}
\put(4074,-12571){\makebox(0,0)[lb]{\smash{{\SetFigFont{10}{12.0}{\rmdefault}{\mddefault}{\updefault}{$cq^3$}%
}}}}
\put(4707,-12558){\makebox(0,0)[lb]{\smash{{\SetFigFont{10}{12.0}{\rmdefault}{\mddefault}{\updefault}{$dq^3$}%
}}}}
\put(1280,-10976){\makebox(0,0)[lb]{\smash{{\SetFigFont{10}{12.0}{\rmdefault}{\mddefault}{\updefault}{$a$}%
}}}}
\put(1947,-10976){\makebox(0,0)[lb]{\smash{{\SetFigFont{10}{12.0}{\rmdefault}{\mddefault}{\updefault}{$b$}%
}}}}
\put(2240,-10961){\makebox(0,0)[lb]{\smash{{\SetFigFont{10}{12.0}{\rmdefault}{\mddefault}{\updefault}{$a$}%
}}}}
\put(2885,-10991){\makebox(0,0)[lb]{\smash{{\SetFigFont{10}{12.0}{\rmdefault}{\mddefault}{\updefault}{$b$}%
}}}}
\put(3147,-10983){\makebox(0,0)[lb]{\smash{{\SetFigFont{10}{12.0}{\rmdefault}{\mddefault}{\updefault}{$a$}%
}}}}
\put(3837,-10983){\makebox(0,0)[lb]{\smash{{\SetFigFont{10}{12.0}{\rmdefault}{\mddefault}{\updefault}{$b$}%
}}}}
\put(4085,-10991){\makebox(0,0)[lb]{\smash{{\SetFigFont{10}{12.0}{\rmdefault}{\mddefault}{\updefault}{$a$}%
}}}}
\put(4805,-11006){\makebox(0,0)[lb]{\smash{{\SetFigFont{10}{12.0}{\rmdefault}{\mddefault}{\updefault}{$b$}%
}}}}
\put(1332,-9888){\makebox(0,0)[lb]{\smash{{\SetFigFont{10}{12.0}{\rmdefault}{\mddefault}{\updefault}{$a$}%
}}}}
\put(1880,-9903){\makebox(0,0)[lb]{\smash{{\SetFigFont{10}{12.0}{\rmdefault}{\mddefault}{\updefault}{$b$}%
}}}}
\put(2307,-9888){\makebox(0,0)[lb]{\smash{{\SetFigFont{10}{12.0}{\rmdefault}{\mddefault}{\updefault}{$a$}%
}}}}
\put(2795,-9926){\makebox(0,0)[lb]{\smash{{\SetFigFont{10}{12.0}{\rmdefault}{\mddefault}{\updefault}{$b$}%
}}}}
\put(3252,-9888){\makebox(0,0)[lb]{\smash{{\SetFigFont{10}{12.0}{\rmdefault}{\mddefault}{\updefault}{$a$}%
}}}}
\put(3770,-9926){\makebox(0,0)[lb]{\smash{{\SetFigFont{10}{12.0}{\rmdefault}{\mddefault}{\updefault}{$b$}%
}}}}
\put(4175,-9903){\makebox(0,0)[lb]{\smash{{\SetFigFont{10}{12.0}{\rmdefault}{\mddefault}{\updefault}{$a$}%
}}}}
\put(4737,-9948){\makebox(0,0)[lb]{\smash{{\SetFigFont{10}{12.0}{\rmdefault}{\mddefault}{\updefault}{$b$}%
}}}}
\put(1302,-11606){\makebox(0,0)[lb]{\smash{{\SetFigFont{10}{12.0}{\rmdefault}{\mddefault}{\updefault}{$cq$}%
}}}}
\put(1797,-11636){\makebox(0,0)[lb]{\smash{{\SetFigFont{10}{12.0}{\rmdefault}{\mddefault}{\updefault}{$dq$}%
}}}}
\put(2202,-11636){\makebox(0,0)[lb]{\smash{{\SetFigFont{10}{12.0}{\rmdefault}{\mddefault}{\updefault}{$cq^2$}%
}}}}
\put(2750,-11621){\makebox(0,0)[lb]{\smash{{\SetFigFont{10}{12.0}{\rmdefault}{\mddefault}{\updefault}{$dq^2$}%
}}}}
\put(3132,-11621){\makebox(0,0)[lb]{\smash{{\SetFigFont{10}{12.0}{\rmdefault}{\mddefault}{\updefault}{$cq^3$}%
}}}}
\put(3702,-11628){\makebox(0,0)[lb]{\smash{{\SetFigFont{10}{12.0}{\rmdefault}{\mddefault}{\updefault}{$dq^3$}%
}}}}
\put(1160,-10653){\makebox(0,0)[lb]{\smash{{\SetFigFont{10}{12.0}{\rmdefault}{\mddefault}{\updefault}{$cq^2$}%
}}}}
\put(1850,-10676){\makebox(0,0)[lb]{\smash{{\SetFigFont{10}{12.0}{\rmdefault}{\mddefault}{\updefault}{$dq^2$}%
}}}}
\put(2202,-10668){\makebox(0,0)[lb]{\smash{{\SetFigFont{10}{12.0}{\rmdefault}{\mddefault}{\updefault}{$cq^3$}%
}}}}
\put(2787,-10668){\makebox(0,0)[lb]{\smash{{\SetFigFont{10}{12.0}{\rmdefault}{\mddefault}{\updefault}{$dq^3$}%
}}}}
\put(3140,-10683){\makebox(0,0)[lb]{\smash{{\SetFigFont{10}{12.0}{\rmdefault}{\mddefault}{\updefault}{$cq^4$}%
}}}}
\put(3717,-10683){\makebox(0,0)[lb]{\smash{{\SetFigFont{10}{12.0}{\rmdefault}{\mddefault}{\updefault}{$dq^4$}%
}}}}
\put(4092,-10698){\makebox(0,0)[lb]{\smash{{\SetFigFont{10}{12.0}{\rmdefault}{\mddefault}{\updefault}{$cq^5$}%
}}}}
\put(4692,-10691){\makebox(0,0)[lb]{\smash{{\SetFigFont{10}{12.0}{\rmdefault}{\mddefault}{\updefault}{$dq^5$}%
}}}}
\put(4077,-11636){\makebox(0,0)[lb]{\smash{{\SetFigFont{10}{12.0}{\rmdefault}{\mddefault}{\updefault}{$cq^4$}%
}}}}
\put(4662,-11628){\makebox(0,0)[lb]{\smash{{\SetFigFont{10}{12.0}{\rmdefault}{\mddefault}{\updefault}{$dq^4$}%
}}}}
\put(8463,-12208){\makebox(0,0)[lb]{\smash{{\SetFigFont{10}{12.0}{\rmdefault}{\mddefault}{\updefault}{$aq$}%
}}}}
\put(9303,-12201){\makebox(0,0)[lb]{\smash{{\SetFigFont{10}{12.0}{\rmdefault}{\mddefault}{\updefault}{$b$}%
}}}}
\put(9490,-12208){\makebox(0,0)[lb]{\smash{{\SetFigFont{10}{12.0}{\rmdefault}{\mddefault}{\updefault}{$aq$}%
}}}}
\put(10233,-12208){\makebox(0,0)[lb]{\smash{{\SetFigFont{10}{12.0}{\rmdefault}{\mddefault}{\updefault}{$b$}%
}}}}
\put(10428,-12201){\makebox(0,0)[lb]{\smash{{\SetFigFont{10}{12.0}{\rmdefault}{\mddefault}{\updefault}{$aq$}%
}}}}
\put(11200,-12193){\makebox(0,0)[lb]{\smash{{\SetFigFont{10}{12.0}{\rmdefault}{\mddefault}{\updefault}{$b$}%
}}}}
\put(8583,-11218){\makebox(0,0)[lb]{\smash{{\SetFigFont{10}{12.0}{\rmdefault}{\mddefault}{\updefault}{$aq$}%
}}}}
\put(9280,-11241){\makebox(0,0)[lb]{\smash{{\SetFigFont{10}{12.0}{\rmdefault}{\mddefault}{\updefault}{$b$}%
}}}}
\put(9490,-11226){\makebox(0,0)[lb]{\smash{{\SetFigFont{10}{12.0}{\rmdefault}{\mddefault}{\updefault}{$aq$}%
}}}}
\put(10263,-11256){\makebox(0,0)[lb]{\smash{{\SetFigFont{10}{12.0}{\rmdefault}{\mddefault}{\updefault}{$b$}%
}}}}
\put(10443,-11248){\makebox(0,0)[lb]{\smash{{\SetFigFont{10}{12.0}{\rmdefault}{\mddefault}{\updefault}{$aq$}%
}}}}
\put(11223,-11263){\makebox(0,0)[lb]{\smash{{\SetFigFont{10}{12.0}{\rmdefault}{\mddefault}{\updefault}{$b$}%
}}}}
\put(8553,-11841){\makebox(0,0)[lb]{\smash{{\SetFigFont{10}{12.0}{\rmdefault}{\mddefault}{\updefault}{$cq$}%
}}}}
\put(9145,-11863){\makebox(0,0)[lb]{\smash{{\SetFigFont{10}{12.0}{\rmdefault}{\mddefault}{\updefault}{$dq$}%
}}}}
\put(9513,-11878){\makebox(0,0)[lb]{\smash{{\SetFigFont{10}{12.0}{\rmdefault}{\mddefault}{\updefault}{$cq^2$}%
}}}}
\put(10075,-11871){\makebox(0,0)[lb]{\smash{{\SetFigFont{10}{12.0}{\rmdefault}{\mddefault}{\updefault}{$dq^2$}%
}}}}
\put(10465,-11878){\makebox(0,0)[lb]{\smash{{\SetFigFont{10}{12.0}{\rmdefault}{\mddefault}{\updefault}{$cq^3$}%
}}}}
\put(11020,-11886){\makebox(0,0)[lb]{\smash{{\SetFigFont{10}{12.0}{\rmdefault}{\mddefault}{\updefault}{$dq^3$}%
}}}}
\put(8568,-10266){\makebox(0,0)[lb]{\smash{{\SetFigFont{10}{12.0}{\rmdefault}{\mddefault}{\updefault}{$aq$}%
}}}}
\put(9280,-10266){\makebox(0,0)[lb]{\smash{{\SetFigFont{10}{12.0}{\rmdefault}{\mddefault}{\updefault}{$b$}%
}}}}
\put(9513,-10258){\makebox(0,0)[lb]{\smash{{\SetFigFont{10}{12.0}{\rmdefault}{\mddefault}{\updefault}{$aq$}%
}}}}
\put(10218,-10273){\makebox(0,0)[lb]{\smash{{\SetFigFont{10}{12.0}{\rmdefault}{\mddefault}{\updefault}{$b$}%
}}}}
\put(10488,-10251){\makebox(0,0)[lb]{\smash{{\SetFigFont{10}{12.0}{\rmdefault}{\mddefault}{\updefault}{$aq$}%
}}}}
\put(11178,-10266){\makebox(0,0)[lb]{\smash{{\SetFigFont{10}{12.0}{\rmdefault}{\mddefault}{\updefault}{$b$}%
}}}}
\put(8515,-10888){\makebox(0,0)[lb]{\smash{{\SetFigFont{10}{12.0}{\rmdefault}{\mddefault}{\updefault}{$cq^2$}%
}}}}
\put(9160,-10911){\makebox(0,0)[lb]{\smash{{\SetFigFont{10}{12.0}{\rmdefault}{\mddefault}{\updefault}{$dq^2$}%
}}}}
\put(9498,-10918){\makebox(0,0)[lb]{\smash{{\SetFigFont{10}{12.0}{\rmdefault}{\mddefault}{\updefault}{$cq^3$}%
}}}}
\put(10098,-10911){\makebox(0,0)[lb]{\smash{{\SetFigFont{10}{12.0}{\rmdefault}{\mddefault}{\updefault}{$dq^3$}%
}}}}
\put(10450,-10926){\makebox(0,0)[lb]{\smash{{\SetFigFont{10}{12.0}{\rmdefault}{\mddefault}{\updefault}{$cq^4$}%
}}}}
\put(11035,-10911){\makebox(0,0)[lb]{\smash{{\SetFigFont{10}{12.0}{\rmdefault}{\mddefault}{\updefault}{$dq^4$}%
}}}}
\put(5053,-11928){\makebox(0,0)[lb]{\smash{{\SetFigFont{10}{12.0}{\rmdefault}{\mddefault}{\updefault}{$a$}%
}}}}
\put(5773,-11951){\makebox(0,0)[lb]{\smash{{\SetFigFont{10}{12.0}{\rmdefault}{\mddefault}{\updefault}{$b$}%
}}}}
\put(5031,-10991){\makebox(0,0)[lb]{\smash{{\SetFigFont{10}{12.0}{\rmdefault}{\mddefault}{\updefault}{$a$}%
}}}}
\put(5751,-11006){\makebox(0,0)[lb]{\smash{{\SetFigFont{10}{12.0}{\rmdefault}{\mddefault}{\updefault}{$b$}%
}}}}
\put(5121,-9903){\makebox(0,0)[lb]{\smash{{\SetFigFont{10}{12.0}{\rmdefault}{\mddefault}{\updefault}{$a$}%
}}}}
\put(5683,-9948){\makebox(0,0)[lb]{\smash{{\SetFigFont{10}{12.0}{\rmdefault}{\mddefault}{\updefault}{$b$}%
}}}}
\put(5049,-10702){\makebox(0,0)[lb]{\smash{{\SetFigFont{10}{12.0}{\rmdefault}{\mddefault}{\updefault}{$cq^6$}%
}}}}
\put(5687,-10679){\makebox(0,0)[lb]{\smash{{\SetFigFont{10}{12.0}{\rmdefault}{\mddefault}{\updefault}{$dq^6$}%
}}}}
\put(5004,-11625){\makebox(0,0)[lb]{\smash{{\SetFigFont{10}{12.0}{\rmdefault}{\mddefault}{\updefault}{$cq^5$}%
}}}}
\put(5702,-11617){\makebox(0,0)[lb]{\smash{{\SetFigFont{10}{12.0}{\rmdefault}{\mddefault}{\updefault}{$dq^5$}%
}}}}
\put(5004,-12541){\makebox(0,0)[lb]{\smash{{\SetFigFont{10}{12.0}{\rmdefault}{\mddefault}{\updefault}{$cq^4$}%
}}}}
\put(5590,-12560){\makebox(0,0)[lb]{\smash{{\SetFigFont{10}{12.0}{\rmdefault}{\mddefault}{\updefault}{$dq^4$}%
}}}}
\put(11389,-12194){\makebox(0,0)[lb]{\smash{{\SetFigFont{10}{12.0}{\rmdefault}{\mddefault}{\updefault}{$aq$}%
}}}}
\put(12161,-12186){\makebox(0,0)[lb]{\smash{{\SetFigFont{10}{12.0}{\rmdefault}{\mddefault}{\updefault}{$b$}%
}}}}
\put(11404,-11241){\makebox(0,0)[lb]{\smash{{\SetFigFont{10}{12.0}{\rmdefault}{\mddefault}{\updefault}{$aq$}%
}}}}
\put(12184,-11256){\makebox(0,0)[lb]{\smash{{\SetFigFont{10}{12.0}{\rmdefault}{\mddefault}{\updefault}{$b$}%
}}}}
\put(11449,-10244){\makebox(0,0)[lb]{\smash{{\SetFigFont{10}{12.0}{\rmdefault}{\mddefault}{\updefault}{$aq$}%
}}}}
\put(12139,-10259){\makebox(0,0)[lb]{\smash{{\SetFigFont{10}{12.0}{\rmdefault}{\mddefault}{\updefault}{$b$}%
}}}}
\put(11422,-11870){\makebox(0,0)[lb]{\smash{{\SetFigFont{10}{12.0}{\rmdefault}{\mddefault}{\updefault}{$cq^4$}%
}}}}
\put(12007,-11847){\makebox(0,0)[lb]{\smash{{\SetFigFont{10}{12.0}{\rmdefault}{\mddefault}{\updefault}{$dq^4$}%
}}}}
\put(11407,-10917){\makebox(0,0)[lb]{\smash{{\SetFigFont{10}{12.0}{\rmdefault}{\mddefault}{\updefault}{$cq^5$}%
}}}}
\put(12029,-10910){\makebox(0,0)[lb]{\smash{{\SetFigFont{10}{12.0}{\rmdefault}{\mddefault}{\updefault}{$dq^5$}%
}}}}
\put(1848,-16265){\makebox(0,0)[lb]{\smash{{\SetFigFont{10}{12.0}{\rmdefault}{\mddefault}{\updefault}{$aq$}%
}}}}
\put(2688,-16258){\makebox(0,0)[lb]{\smash{{\SetFigFont{10}{12.0}{\rmdefault}{\mddefault}{\updefault}{$b$}%
}}}}
\put(2875,-16265){\makebox(0,0)[lb]{\smash{{\SetFigFont{10}{12.0}{\rmdefault}{\mddefault}{\updefault}{$aq$}%
}}}}
\put(3618,-16265){\makebox(0,0)[lb]{\smash{{\SetFigFont{10}{12.0}{\rmdefault}{\mddefault}{\updefault}{$b$}%
}}}}
\put(3813,-16258){\makebox(0,0)[lb]{\smash{{\SetFigFont{10}{12.0}{\rmdefault}{\mddefault}{\updefault}{$aq$}%
}}}}
\put(4585,-16250){\makebox(0,0)[lb]{\smash{{\SetFigFont{10}{12.0}{\rmdefault}{\mddefault}{\updefault}{$b$}%
}}}}
\put(4774,-16251){\makebox(0,0)[lb]{\smash{{\SetFigFont{10}{12.0}{\rmdefault}{\mddefault}{\updefault}{$aq$}%
}}}}
\put(5546,-16243){\makebox(0,0)[lb]{\smash{{\SetFigFont{10}{12.0}{\rmdefault}{\mddefault}{\updefault}{$b$}%
}}}}
\put(2322,-15540){\makebox(0,0)[lb]{\smash{{\SetFigFont{10}{12.0}{\rmdefault}{\mddefault}{\updefault}{$aq^2$}%
}}}}
\put(3139,-15548){\makebox(0,0)[lb]{\smash{{\SetFigFont{10}{12.0}{\rmdefault}{\mddefault}{\updefault}{$b$}%
}}}}
\put(3334,-15525){\makebox(0,0)[lb]{\smash{{\SetFigFont{10}{12.0}{\rmdefault}{\mddefault}{\updefault}{$aq^2$}%
}}}}
\put(4092,-15525){\makebox(0,0)[lb]{\smash{{\SetFigFont{10}{12.0}{\rmdefault}{\mddefault}{\updefault}{$b$}%
}}}}
\put(5037,-15525){\makebox(0,0)[lb]{\smash{{\SetFigFont{10}{12.0}{\rmdefault}{\mddefault}{\updefault}{$b$}%
}}}}
\put(4279,-15525){\makebox(0,0)[lb]{\smash{{\SetFigFont{10}{12.0}{\rmdefault}{\mddefault}{\updefault}{$aq^2$}%
}}}}
\put(2329,-15158){\makebox(0,0)[lb]{\smash{{\SetFigFont{10}{12.0}{\rmdefault}{\mddefault}{\updefault}{$cq^2$}%
}}}}
\put(3030,-15202){\makebox(0,0)[lb]{\smash{{\SetFigFont{10}{12.0}{\rmdefault}{\mddefault}{\updefault}{$dq^2$}%
}}}}
\put(3319,-15188){\makebox(0,0)[lb]{\smash{{\SetFigFont{10}{12.0}{\rmdefault}{\mddefault}{\updefault}{$cq^3$}%
}}}}
\put(3968,-15202){\makebox(0,0)[lb]{\smash{{\SetFigFont{10}{12.0}{\rmdefault}{\mddefault}{\updefault}{$dq^3$}%
}}}}
\put(4264,-15195){\makebox(0,0)[lb]{\smash{{\SetFigFont{10}{12.0}{\rmdefault}{\mddefault}{\updefault}{$cq^4$}%
}}}}
\put(4905,-15202){\makebox(0,0)[lb]{\smash{{\SetFigFont{10}{12.0}{\rmdefault}{\mddefault}{\updefault}{$dq^4$}%
}}}}
\put(2374,-14520){\makebox(0,0)[lb]{\smash{{\SetFigFont{10}{12.0}{\rmdefault}{\mddefault}{\updefault}{$aq^2$}%
}}}}
\put(3049,-14520){\makebox(0,0)[lb]{\smash{{\SetFigFont{10}{12.0}{\rmdefault}{\mddefault}{\updefault}{$b$}%
}}}}
\put(3334,-14490){\makebox(0,0)[lb]{\smash{{\SetFigFont{10}{12.0}{\rmdefault}{\mddefault}{\updefault}{$aq^2$}%
}}}}
\put(4047,-14513){\makebox(0,0)[lb]{\smash{{\SetFigFont{10}{12.0}{\rmdefault}{\mddefault}{\updefault}{$b$}%
}}}}
\put(4984,-14483){\makebox(0,0)[lb]{\smash{{\SetFigFont{10}{12.0}{\rmdefault}{\mddefault}{\updefault}{$b$}%
}}}}
\put(4287,-14490){\makebox(0,0)[lb]{\smash{{\SetFigFont{10}{12.0}{\rmdefault}{\mddefault}{\updefault}{$aq^2$}%
}}}}
\put(7997,-16246){\makebox(0,0)[lb]{\smash{{\SetFigFont{10}{12.0}{\rmdefault}{\mddefault}{\updefault}{$aq$}%
}}}}
\put(8843,-16238){\makebox(0,0)[lb]{\smash{{\SetFigFont{10}{12.0}{\rmdefault}{\mddefault}{\updefault}{$b$}%
}}}}
\put(9030,-16245){\makebox(0,0)[lb]{\smash{{\SetFigFont{10}{12.0}{\rmdefault}{\mddefault}{\updefault}{$aq$}%
}}}}
\put(9773,-16245){\makebox(0,0)[lb]{\smash{{\SetFigFont{10}{12.0}{\rmdefault}{\mddefault}{\updefault}{$b$}%
}}}}
\put(9968,-16238){\makebox(0,0)[lb]{\smash{{\SetFigFont{10}{12.0}{\rmdefault}{\mddefault}{\updefault}{$aq$}%
}}}}
\put(10740,-16230){\makebox(0,0)[lb]{\smash{{\SetFigFont{10}{12.0}{\rmdefault}{\mddefault}{\updefault}{$b$}%
}}}}
\put(10929,-16231){\makebox(0,0)[lb]{\smash{{\SetFigFont{10}{12.0}{\rmdefault}{\mddefault}{\updefault}{$aq$}%
}}}}
\put(11701,-16223){\makebox(0,0)[lb]{\smash{{\SetFigFont{10}{12.0}{\rmdefault}{\mddefault}{\updefault}{$b$}%
}}}}
\put(8486,-15494){\makebox(0,0)[lb]{\smash{{\SetFigFont{10}{12.0}{\rmdefault}{\mddefault}{\updefault}{$aq^2$}%
}}}}
\put(9273,-15494){\makebox(0,0)[lb]{\smash{{\SetFigFont{10}{12.0}{\rmdefault}{\mddefault}{\updefault}{$b$}%
}}}}
\put(9461,-15479){\makebox(0,0)[lb]{\smash{{\SetFigFont{10}{12.0}{\rmdefault}{\mddefault}{\updefault}{$aq^2$}%
}}}}
\put(10248,-15494){\makebox(0,0)[lb]{\smash{{\SetFigFont{10}{12.0}{\rmdefault}{\mddefault}{\updefault}{$b$}%
}}}}
\put(11208,-15486){\makebox(0,0)[lb]{\smash{{\SetFigFont{10}{12.0}{\rmdefault}{\mddefault}{\updefault}{$b$}%
}}}}
\put(10436,-15486){\makebox(0,0)[lb]{\smash{{\SetFigFont{10}{12.0}{\rmdefault}{\mddefault}{\updefault}{$aq^2$}%
}}}}
\put(8996,-14774){\makebox(0,0)[lb]{\smash{{\SetFigFont{10}{12.0}{\rmdefault}{\mddefault}{\updefault}{$aq^3$}%
}}}}
\put(9708,-14766){\makebox(0,0)[lb]{\smash{{\SetFigFont{10}{12.0}{\rmdefault}{\mddefault}{\updefault}{$b$}%
}}}}
\put(9978,-14751){\makebox(0,0)[lb]{\smash{{\SetFigFont{10}{12.0}{\rmdefault}{\mddefault}{\updefault}{$aq^3$}%
}}}}
\put(10721,-14766){\makebox(0,0)[lb]{\smash{{\SetFigFont{10}{12.0}{\rmdefault}{\mddefault}{\updefault}{$b$}%
}}}}
\put(4927,-19106){\makebox(0,0)[lb]{\smash{{\SetFigFont{10}{12.0}{\rmdefault}{\mddefault}{\updefault}{$aq$}%
}}}}
\put(7664,-19091){\makebox(0,0)[lb]{\smash{{\SetFigFont{10}{12.0}{\rmdefault}{\mddefault}{\updefault}{$b$}%
}}}}
\put(7853,-19092){\makebox(0,0)[lb]{\smash{{\SetFigFont{10}{12.0}{\rmdefault}{\mddefault}{\updefault}{$aq$}%
}}}}
\put(5410,-18355){\makebox(0,0)[lb]{\smash{{\SetFigFont{10}{12.0}{\rmdefault}{\mddefault}{\updefault}{$aq^2$}%
}}}}
\put(6197,-18355){\makebox(0,0)[lb]{\smash{{\SetFigFont{10}{12.0}{\rmdefault}{\mddefault}{\updefault}{$b$}%
}}}}
\put(6385,-18340){\makebox(0,0)[lb]{\smash{{\SetFigFont{10}{12.0}{\rmdefault}{\mddefault}{\updefault}{$aq^2$}%
}}}}
\put(7172,-18355){\makebox(0,0)[lb]{\smash{{\SetFigFont{10}{12.0}{\rmdefault}{\mddefault}{\updefault}{$b$}%
}}}}
\put(8132,-18347){\makebox(0,0)[lb]{\smash{{\SetFigFont{10}{12.0}{\rmdefault}{\mddefault}{\updefault}{$b$}%
}}}}
\put(7360,-18347){\makebox(0,0)[lb]{\smash{{\SetFigFont{10}{12.0}{\rmdefault}{\mddefault}{\updefault}{$aq^2$}%
}}}}
\put(5920,-17635){\makebox(0,0)[lb]{\smash{{\SetFigFont{10}{12.0}{\rmdefault}{\mddefault}{\updefault}{$aq^3$}%
}}}}
\put(6632,-17627){\makebox(0,0)[lb]{\smash{{\SetFigFont{10}{12.0}{\rmdefault}{\mddefault}{\updefault}{$b$}%
}}}}
\put(6902,-17612){\makebox(0,0)[lb]{\smash{{\SetFigFont{10}{12.0}{\rmdefault}{\mddefault}{\updefault}{$aq^3$}%
}}}}
\put(7645,-17627){\makebox(0,0)[lb]{\smash{{\SetFigFont{10}{12.0}{\rmdefault}{\mddefault}{\updefault}{$b$}%
}}}}
\end{picture}}
\caption{Transforming an Aztec rectangle graph into the dual graph of a semi-hexagon.}
\label{holeytransform2}
\end{figure}

Let $T$ be any lozenge tiling of $\mathcal{SH}_{m,n-m}(s_1,\dotsc,s_m)$. Encode $T$ as a family of $n-m$ disjoint rhombi-paths connecting the top and the bottom of the region as in Figures \ref{rhombuspath}(a) and (b). This implies that
 \begin{equation}
  \M(\widetilde{G})=\sum_{\mathbf{P}=(P_1,\dotsc,P_{n-m})}\prod_{i=1}^{n-m}\wt(P_i),
 \end{equation}
 where the sum is taken over all families of disjoint rhombi-paths $\mathbf{P}=(P_1,P_2,\dotsc,P_{n-m})$ connecting the top and the bottom of the region, and where $\wt(P_i)$ is the product of weights of all rhombi in $P_i$.

 \begin{figure}\centering
\resizebox{!}{10cm}{\begin{picture}(0,0)%
\includegraphics{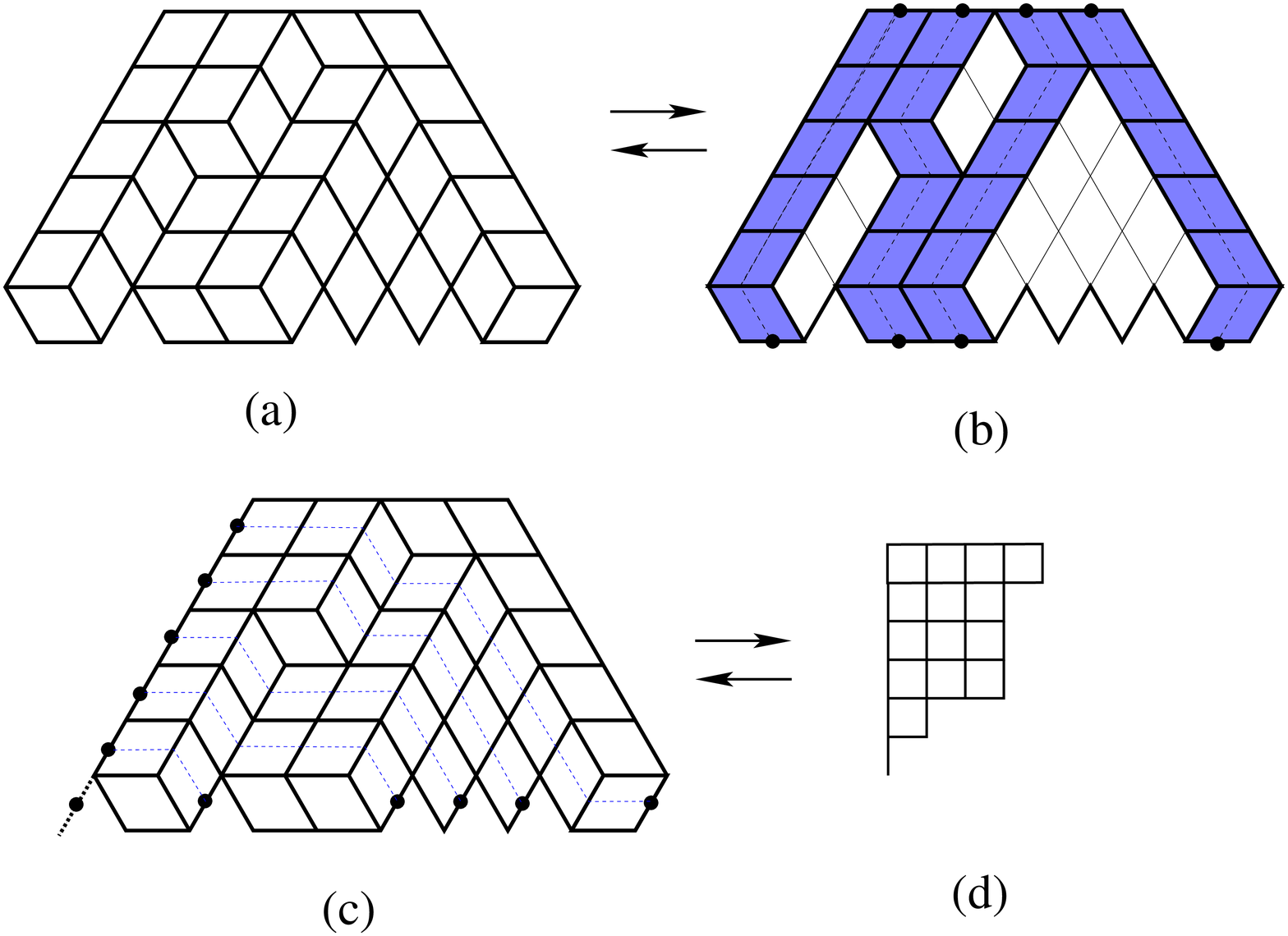}%
\end{picture}%
%
%
\setlength{\unitlength}{3947sp}%
\begingroup\makeatletter\ifx\SetFigFont\undefined%
\gdef\SetFigFont#1#2#3#4#5{%
  \reset@font\fontsize{#1}{#2pt}%
  \fontfamily{#3}\fontseries{#4}\fontshape{#5}%
  \selectfont}%
\fi\endgroup%
\begin{picture}(15709,11848)(346,-11428)
\put(11356,-6946){\makebox(0,0)[lb]{\smash{{\SetFigFont{20}{24.0}{\rmdefault}{\mddefault}{\updefault}{\color[rgb]{0,0,0}6}%
}}}}
\put(11356,-7891){\makebox(0,0)[lb]{\smash{{\SetFigFont{20}{24.0}{\rmdefault}{\mddefault}{\updefault}{\color[rgb]{0,0,0}4}%
}}}}
\put(5034,-7779){\makebox(0,0)[lb]{\smash{{\SetFigFont{20}{24.0}{\rmdefault}{\mddefault}{\updefault}{\color[rgb]{1,0,0}$q^4$}%
}}}}
\put(12286,-7441){\makebox(0,0)[lb]{\smash{{\SetFigFont{20}{24.0}{\rmdefault}{\mddefault}{\updefault}{\color[rgb]{0,0,0}4}%
}}}}
\put(11836,-7411){\makebox(0,0)[lb]{\smash{{\SetFigFont{20}{24.0}{\rmdefault}{\mddefault}{\updefault}{\color[rgb]{0,0,0}5}%
}}}}
\put(11356,-7396){\makebox(0,0)[lb]{\smash{{\SetFigFont{20}{24.0}{\rmdefault}{\mddefault}{\updefault}{\color[rgb]{0,0,0}5}%
}}}}
\put(12766,-6961){\makebox(0,0)[lb]{\smash{{\SetFigFont{20}{24.0}{\rmdefault}{\mddefault}{\updefault}{\color[rgb]{0,0,0}1}%
}}}}
\put(12271,-6931){\makebox(0,0)[lb]{\smash{{\SetFigFont{20}{24.0}{\rmdefault}{\mddefault}{\updefault}{\color[rgb]{0,0,0}5}%
}}}}
\put(11821,-6946){\makebox(0,0)[lb]{\smash{{\SetFigFont{20}{24.0}{\rmdefault}{\mddefault}{\updefault}{\color[rgb]{0,0,0}6}%
}}}}
\put(2581,-6346){\makebox(0,0)[lb]{\smash{{\SetFigFont{20}{24.0}{\rmdefault}{\mddefault}{\itdefault}{\color[rgb]{0,0,0}$Q_1$}%
}}}}
\put(2191,-7081){\makebox(0,0)[lb]{\smash{{\SetFigFont{20}{24.0}{\rmdefault}{\mddefault}{\itdefault}{\color[rgb]{0,0,0}$Q_2$}%
}}}}
\put(1756,-7771){\makebox(0,0)[lb]{\smash{{\SetFigFont{20}{24.0}{\rmdefault}{\mddefault}{\itdefault}{\color[rgb]{0,0,0}$Q_3$}%
}}}}
\put(1426,-8386){\makebox(0,0)[lb]{\smash{{\SetFigFont{20}{24.0}{\rmdefault}{\mddefault}{\itdefault}{\color[rgb]{0,0,0}$Q_4$}%
}}}}
\put(1081,-9031){\makebox(0,0)[lb]{\smash{{\SetFigFont{20}{24.0}{\rmdefault}{\mddefault}{\itdefault}{\color[rgb]{0,0,0}$Q_5$}%
}}}}
\put(691,-9676){\makebox(0,0)[lb]{\smash{{\SetFigFont{20}{24.0}{\rmdefault}{\mddefault}{\itdefault}{\color[rgb]{0,0,0}$Q_6$}%
}}}}
\put(11191,134){\makebox(0,0)[lb]{\smash{{\SetFigFont{20}{24.0}{\rmdefault}{\bfdefault}{\updefault}{\color[rgb]{0,0,0}$P_1$}%
}}}}
\put(11941,134){\makebox(0,0)[lb]{\smash{{\SetFigFont{20}{24.0}{\rmdefault}{\bfdefault}{\updefault}{\color[rgb]{0,0,0}$P_2$}%
}}}}
\put(12721,119){\makebox(0,0)[lb]{\smash{{\SetFigFont{20}{24.0}{\rmdefault}{\bfdefault}{\updefault}{\color[rgb]{0,0,0}$P_3$}%
}}}}
\put(13426,149){\makebox(0,0)[lb]{\smash{{\SetFigFont{20}{24.0}{\rmdefault}{\bfdefault}{\updefault}{\color[rgb]{0,0,0}$P_4$}%
}}}}
\put(7846,-9781){\makebox(0,0)[lb]{\smash{{\SetFigFont{20}{24.0}{\rmdefault}{\mddefault}{\updefault}{\color[rgb]{1,0,0}$q$}%
}}}}
\put(4291,-9136){\makebox(0,0)[lb]{\smash{{\SetFigFont{20}{24.0}{\rmdefault}{\mddefault}{\updefault}{\color[rgb]{1,0,0}$q^2$}%
}}}}
\put(3631,-9136){\makebox(0,0)[lb]{\smash{{\SetFigFont{20}{24.0}{\rmdefault}{\mddefault}{\updefault}{\color[rgb]{1,0,0}$q^2$}%
}}}}
\put(1966,-9121){\makebox(0,0)[lb]{\smash{{\SetFigFont{20}{24.0}{\rmdefault}{\mddefault}{\updefault}{\color[rgb]{1,0,0}$q^2$}%
}}}}
\put(2341,-8371){\makebox(0,0)[lb]{\smash{{\SetFigFont{20}{24.0}{\rmdefault}{\mddefault}{\updefault}{\color[rgb]{1,0,0}$q^3$}%
}}}}
\put(3871,-8416){\makebox(0,0)[lb]{\smash{{\SetFigFont{20}{24.0}{\rmdefault}{\mddefault}{\updefault}{\color[rgb]{1,0,0}$q^3$}%
}}}}
\put(4726,-8431){\makebox(0,0)[lb]{\smash{{\SetFigFont{20}{24.0}{\rmdefault}{\mddefault}{\updefault}{\color[rgb]{1,0,0}$q^3$}%
}}}}
\put(2746,-7741){\makebox(0,0)[lb]{\smash{{\SetFigFont{20}{24.0}{\rmdefault}{\mddefault}{\updefault}{\color[rgb]{1,0,0}$q^4$}%
}}}}
\put(3151,-7051){\makebox(0,0)[lb]{\smash{{\SetFigFont{20}{24.0}{\rmdefault}{\mddefault}{\updefault}{\color[rgb]{1,0,0}$q^5$}%
}}}}
\put(3976,-7036){\makebox(0,0)[lb]{\smash{{\SetFigFont{20}{24.0}{\rmdefault}{\mddefault}{\updefault}{\color[rgb]{1,0,0}$q^5$}%
}}}}
\put(5506,-7111){\makebox(0,0)[lb]{\smash{{\SetFigFont{20}{24.0}{\rmdefault}{\mddefault}{\updefault}{\color[rgb]{1,0,0}$q^5$}%
}}}}
\put(4306,-6376){\makebox(0,0)[lb]{\smash{{\SetFigFont{20}{24.0}{\rmdefault}{\mddefault}{\updefault}{\color[rgb]{1,0,0}$q^6$}%
}}}}
\put(3616,-6361){\makebox(0,0)[lb]{\smash{{\SetFigFont{20}{24.0}{\rmdefault}{\mddefault}{\updefault}{\color[rgb]{1,0,0}$q^6$}%
}}}}
\put(11326,-8851){\makebox(0,0)[lb]{\smash{{\SetFigFont{20}{24.0}{\rmdefault}{\mddefault}{\updefault}{\color[rgb]{0,0,0}2}%
}}}}
\put(12301,-8356){\makebox(0,0)[lb]{\smash{{\SetFigFont{20}{24.0}{\rmdefault}{\mddefault}{\updefault}{\color[rgb]{0,0,0}2}%
}}}}
\put(11821,-8371){\makebox(0,0)[lb]{\smash{{\SetFigFont{20}{24.0}{\rmdefault}{\mddefault}{\updefault}{\color[rgb]{0,0,0}2}%
}}}}
\put(11356,-8356){\makebox(0,0)[lb]{\smash{{\SetFigFont{20}{24.0}{\rmdefault}{\mddefault}{\updefault}{\color[rgb]{0,0,0}3}%
}}}}
\put(12316,-7891){\makebox(0,0)[lb]{\smash{{\SetFigFont{20}{24.0}{\rmdefault}{\mddefault}{\updefault}{\color[rgb]{0,0,0}3}%
}}}}
\put(11821,-7906){\makebox(0,0)[lb]{\smash{{\SetFigFont{20}{24.0}{\rmdefault}{\mddefault}{\updefault}{\color[rgb]{0,0,0}3}%
}}}}
\end{picture}}
\caption{(a) and (b). Bijection between tilings of a semihexagon $\mathcal{SH}_{6,4}(1,3,6,7,8,10)$ and families of disjoint rhombi-paths. (c) and (d). Bijection between tilings of a semihexagon $\mathcal{SH}_{6,4}(1,3,6,7,8,10)$ and column-strict plane partitions of shape $(10-6,8-5,7-4,6-3,3-2,1-1)$ with positive entries at most $6$.}
\label{rhombuspath}
\end{figure}

 Next, we change the weights of the rhombi as follows. Reassign each right rhombus a weight $1$, and divide the weight of each left rhombus by  $a$. Denote by $\wt'$ the new weight function. We have $\wt(P_i)=b^{m-s_i+i}a^{s_i-i}\wt'(P_i)$ (since each rhombi-path $P_i$ has exactly $s_i-i$ left rhombi and $m-s_i+i$ right rhombi). Denote by $\overline{G}$ the resulting weighted version of $\widetilde{G}$, and $\overline{\mathcal{SH}}$ the corresponding weighted version of $\mathcal{SH}_{m,n-m}(s_1,\dotsc,s_m)$. We get
  \begin{align}\label{Gp2}
  \M(\widetilde{G})&=b^{\sum_{i=1}^{n-m}(m-s_i+i)}a^{\sum_{i=1}^{n-m}(s_i-i)}\sum_{\mathbf{P}
  =(P_1,\dotsc,P_{n-m})}\prod_{i=1}^{n-m}\wt'(P_i)\notag\\
  &=b^{m(n-m)-\sum_{i=1}^{m}(s_i-i)}a^{\sum_{i=1}^{m}(s_i-i)} \M(\overline{G}).
 \end{align}
Now, all left and vertical rhombi in $\overline{\mathcal{SH}}$ are weighted by $1$, and each right rhombus on level $k$ is weighted by $q^{k+1}$ (see Figure \ref{rhombuspath}(c)).

We now encode each lozenge tiling $T$ of $\overline{\mathcal{SH}}$ as a $m$-tuple of (new) disjoint rhombi-paths $(Q_1,\dotsc,Q_{m})$ connecting the northwest side and the left sides of the dents (illustrated in Figure \ref{rhombuspath}(c)).  Some of the $Q_i$ paths may be empty (when $s_i=i$). The exponents of $q$ along the path $Q_i$ gives the entries of the $i$-th row of a \emph{column-strict plane partition} of shape $(s_m-m, s_{m-1}-m+1, \dotsc , s_1-1)$ with positive entries at most $m$ (see Figure \ref{rhombuspath}(d)). We note that path $Q_6$ in Figure \ref{rhombuspath}(c) is empty; and the vertical interval at the bottom of the plane partition in Figure \ref{rhombuspath}(d) presents a row of length $0$. It is easy to verify that the above correspondence yields a bijection. Moreover, the weight of the tiling $T$ of $\overline{\mathcal{SH}}$ is exactly $q^{|\pi_T|}$, where $\pi_T$ is the column-strict plane partition corresponding  to $T$ and where $|\pi_T|$ is the sum of all entries of $\pi_T$. 

Summing over all tilings $T$ of $\overline{\mathcal{SH}}$, we have
\begin{equation}\label{Gp}
\M(\overline{G})=\M(\overline{\mathcal{SH}})=\sum_{\pi}q^{|\pi|}=q^{\sum_{i=1}^{m}(s_i-i)}\prod_{1\leq i<j\leq m}\frac{q^{s_j}-q^{s_i}}{q^{j}-q^{i}},
\end{equation}
where the sum after the first equal sign is taken over all column-strict plane partitions $\pi$ of shape $(s_m-m, s_{m-1}-m+1, \dotsc, s_1-1)$ with positive entries at most $m$; for the second equal sign see e.g. \cite{Stanley}, pp. 375. Then the theorem follows from (\ref{masstransform}), (\ref{Gp2}) and (\ref{Gp}).
\end{proof}

\begin{rmk}
For $1\leq s_1<s_2<\dotsc<s_m\leq n$, we have the following identity
\begin{align}\label{relation}
\M(\mathcal{AR}_{m,n}(s_1,s_2,\dotsc,s_{m})=2^{\frac{m(m+1)}{2}}\M(\mathcal{SH}_{m,n-m}(s_1,s_2,\dotsc,s_{m})),
\end{align}
where the numbers of tilings on both sides are equal to $2^{\frac{m(m+1)}{2}}\prod_{i<j}\frac{s_j-s_i}{j-i}$ (see e.g.  Lemma 3 in \cite{Gessel} and Proposition 2.1 in \cite{Cohn}).  This gives an interesting connection between two \emph{different} types of tilings: domino tiling on the left-hand side and lozenge tiling on the right-hand side. By letting $a=b=c=d=q=1$, Figure \ref{holeytransform2} gives a (simple) combinatorial explanation for the relation (\ref{relation}).  Moreover, our explanation is direct in the sense that it does not require any explicit enumeration of tilings of the two regions.
\end{rmk}

\section{Proof of Theorem \ref{main}}

A \textit{Schr\"{o}der path} is a lattice path on $\mathbb{Z}^2$, starting and ending on the $x$-axis, never going below the $x$-axis, using $(1, 1)$, $(1,-1)$ and $(2,0)$ steps
 (i.e. up, down and  level steps, respectively). See Figure \ref{schrodernobarrier} for a Schr\"{o}der path.

\begin{figure}\centering
\resizebox{!}{3cm}{
\begin{picture}(0,0)%
\includegraphics{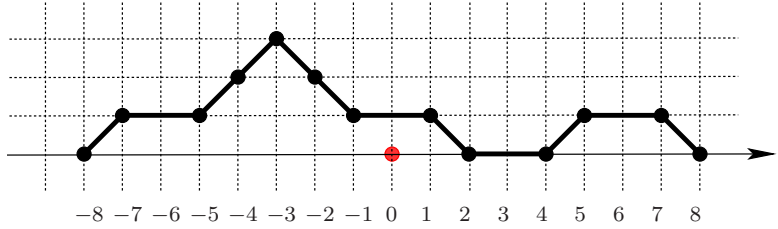}%
\end{picture}%
\setlength{\unitlength}{3947sp}%
\begingroup\makeatletter\ifx\SetFigFont\undefined%
\gdef\SetFigFont#1#2#3#4#5{%
  \reset@font\fontsize{#1}{#2pt}%
  \fontfamily{#3}\fontseries{#4}\fontshape{#5}%
  \selectfont}%
\fi\endgroup%
\begin{picture}(4748,1385)(1643,-4550)
\put(3980,-4535){\makebox(0,0)[lb]{\smash{{\SetFigFont{8}{12.0}{\rmdefault}{\mddefault}{\updefault}{$0$}%
}}}}
\put(4194,-4535){\makebox(0,0)[lb]{\smash{{\SetFigFont{8}{12.0}{\rmdefault}{\mddefault}{\updefault}{$1$}%
}}}}
\put(4430,-4535){\makebox(0,0)[lb]{\smash{{\SetFigFont{8}{12.0}{\rmdefault}{\mddefault}{\updefault}{$2$}%
}}}}
\put(4666,-4535){\makebox(0,0)[lb]{\smash{{\SetFigFont{8}{12.0}{\rmdefault}{\mddefault}{\updefault}{$3$}%
}}}}
\put(4903,-4535){\makebox(0,0)[lb]{\smash{{\SetFigFont{8}{12.0}{\rmdefault}{\mddefault}{\updefault}{$4$}%
}}}}
\put(5139,-4535){\makebox(0,0)[lb]{\smash{{\SetFigFont{8}{12.0}{\rmdefault}{\mddefault}{\updefault}{$5$}%
}}}}
\put(5375,-4535){\makebox(0,0)[lb]{\smash{{\SetFigFont{8}{12.0}{\rmdefault}{\mddefault}{\updefault}{$6$}%
}}}}
\put(5611,-4535){\makebox(0,0)[lb]{\smash{{\SetFigFont{8}{12.0}{\rmdefault}{\mddefault}{\updefault}{$7$}%
}}}}
\put(5847,-4535){\makebox(0,0)[lb]{\smash{{\SetFigFont{8}{12.0}{\rmdefault}{\mddefault}{\updefault}{$8$}%
}}}}
\put(3722,-4535){\makebox(0,0)[lb]{\smash{{\SetFigFont{8}{12.0}{\rmdefault}{\mddefault}{\updefault}{$-1$}%
}}}}
\put(3485,-4535){\makebox(0,0)[lb]{\smash{{\SetFigFont{8}{12.0}{\rmdefault}{\mddefault}{\updefault}{$-2$}%
}}}}
\put(3249,-4535){\makebox(0,0)[lb]{\smash{{\SetFigFont{8}{12.0}{\rmdefault}{\mddefault}{\updefault}{$-3$}%
}}}}
\put(3013,-4535){\makebox(0,0)[lb]{\smash{{\SetFigFont{8}{12.0}{\rmdefault}{\mddefault}{\updefault}{$-4$}%
}}}}
\put(2777,-4535){\makebox(0,0)[lb]{\smash{{\SetFigFont{8}{12.0}{\rmdefault}{\mddefault}{\updefault}{$-5$}%
}}}}
\put(2540,-4535){\makebox(0,0)[lb]{\smash{{\SetFigFont{8}{12.0}{\rmdefault}{\mddefault}{\updefault}{$-6$}%
}}}}
\put(2304,-4535){\makebox(0,0)[lb]{\smash{{\SetFigFont{8}{12.0}{\rmdefault}{\mddefault}{\updefault}{$-7$}%
}}}}
\put(2068,-4535){\makebox(0,0)[lb]{\smash{{\SetFigFont{8}{12.0}{\rmdefault}{\mddefault}{\updefault}{$-8$}%
}}}}
\end{picture}}
\caption{A Schr\"{o}der path from $(-8,0)$ to $(8,0)$.}
\label{schrodernobarrier}
\end{figure}

Color the Aztec rectangle with holes $\mathcal{AR}_{m,n}(s_1,\dotsc,s_m)$ by black and white so that two adjacent unit squares have opposite color and  that the unit squares along the northwest side are white. Decorating the dominoes of the region as in Figure \ref{elementmove}, we have also a bijection between  the tilings of the region and families of non-intersecting (partial shifted) Schr\"{o}der paths $\textbf{P}=(P_1,P_2,\dotsc,P_m)$, where $P_i$ connects the the centers of $i$-th vertical steps on the southwest and the southeast boundaries of the region (see Figure \ref{Holeyschroder}). 

\begin{figure}\centering
\includegraphics{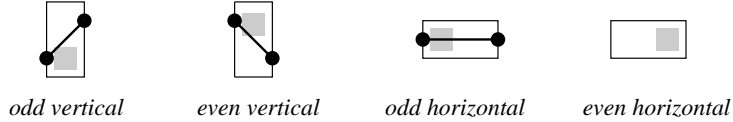}
\caption{Drawing the steps of the Schr\"{o}der paths.}
\label{elementmove}
\end{figure}

\begin{figure}\centering
\resizebox{!}{6.5cm}{
\begin{picture}(0,0)%
\includegraphics{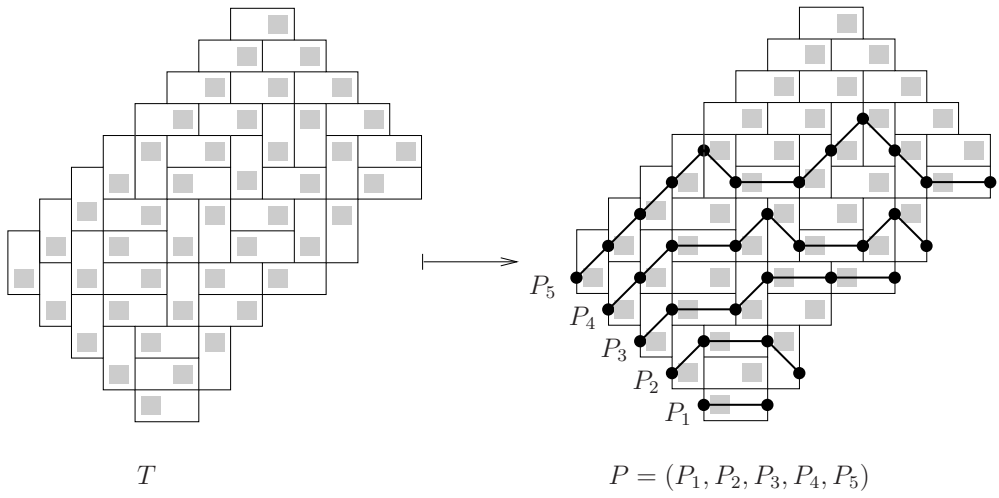}%
\end{picture}%
%
%
\setlength{\unitlength}{3947sp}%
\begingroup\makeatletter\ifx\SetFigFont\undefined%
\gdef\SetFigFont#1#2#3#4#5{%
  \reset@font\fontsize{#1}{#2pt}%
  \fontfamily{#3}\fontseries{#4}\fontshape{#5}%
  \selectfont}%
\fi\endgroup%
\begin{picture}(7356,3629)(1202,-4903)
\put(2174,-4830){\makebox(0,0)[lb]{\smash{{\SetFigFont{12}{14.4}{\rmdefault}{\mddefault}{\itdefault}{$T$}%
}}}}
\put(5671,-4830){\makebox(0,0)[lb]{\smash{{\SetFigFont{12}{14.4}{\rmdefault}{\mddefault}{\itdefault}{$P=(P_1,P_2,P_3,P_4,P_5)$}%
}}}}
\put(6084,-4358){\makebox(0,0)[lb]{\smash{{\SetFigFont{12}{14.4}{\rmdefault}{\mddefault}{\itdefault}{$P_1$}%
}}}}
\put(5847,-4122){\makebox(0,0)[lb]{\smash{{\SetFigFont{12}{14.4}{\rmdefault}{\mddefault}{\itdefault}{$P_2$}%
}}}}
\put(5611,-3885){\makebox(0,0)[lb]{\smash{{\SetFigFont{12}{14.4}{\rmdefault}{\mddefault}{\itdefault}{$P_3$}%
}}}}
\put(5375,-3649){\makebox(0,0)[lb]{\smash{{\SetFigFont{12}{14.4}{\rmdefault}{\mddefault}{\itdefault}{$P_4$}%
}}}}
\put(5080,-3413){\makebox(0,0)[lb]{\smash{{\SetFigFont{12}{14.4}{\rmdefault}{\mddefault}{\itdefault}{$P_5$}%
}}}}
\end{picture}}
\caption{Bijection between tilings of an Aztec rectangle with holes and families of non-intersection Schr\"{o}der paths.}
\label{Holeyschroder}
\end{figure}

 Assume that the endpoints of $P_1$ are on the $x$-axis. Denote by $\area(P_i)$ is the area underneath $P_i$ (i.e. the area restricted by $P_i$ and the $x$-axis), and define $\area(\textbf{P}):=\sum_{i=1}^{m}\area(P_i)$.  It is easy to see that the family of  Schr\"{o}der paths corresponding to the minimal tiling $T_0$ has the smallest (total) underneath area (see Figure \ref{mintiling}).

\begin{figure}\centering
\resizebox{!}{6.5cm}{
\begin{picture}(0,0)%
\includegraphics{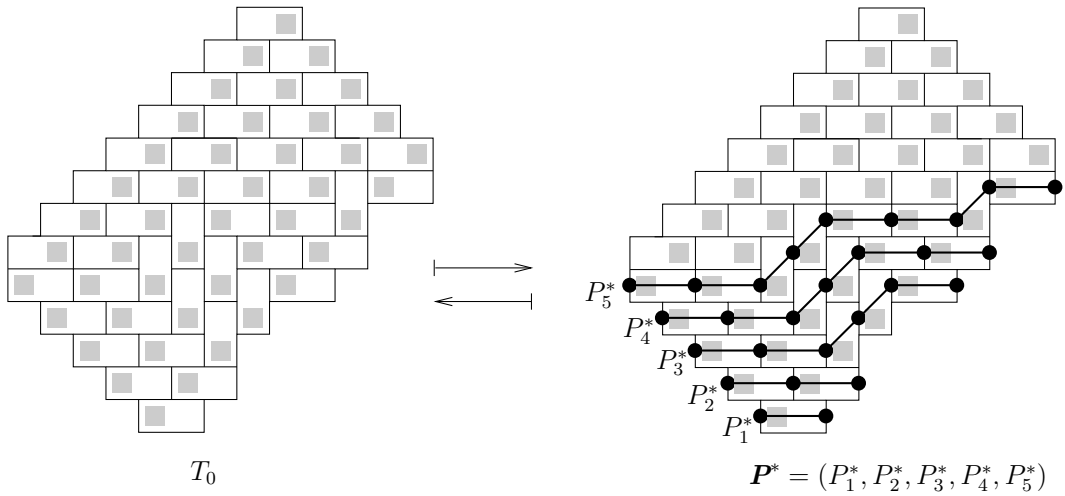}%
\end{picture}
\setlength{\unitlength}{3947sp}%
\begingroup\makeatletter\ifx\SetFigFont\undefined%
\gdef\SetFigFont#1#2#3#4#5{%
  \reset@font\fontsize{#1}{#2pt}%
  \fontfamily{#3}\fontseries{#4}\fontshape{#5}%
  \selectfont}%
\fi\endgroup%
\begin{picture}(7633,3530)(697,-3624)
\put(4810,-2229){\makebox(0,0)[lb]{\smash{{\SetFigFont{12}{14.4}{\rmdefault}{\mddefault}{\itdefault}{\color[rgb]{0,0,0}$P^*_5$}%
}}}}
\put(1966,-3519){\makebox(0,0)[lb]{\smash{{\SetFigFont{12}{14.4}{\rmdefault}{\mddefault}{\itdefault}{\color[rgb]{0,0,0}$T_0$}%
}}}}
\put(5994,-3549){\makebox(0,0)[lb]{\smash{{\SetFigFont{12}{14.4}{\rmdefault}{\mddefault}{\itdefault}{\color[rgb]{0,0,0}$\textbf{P}^*=(P^*_1,P^*_2,P^*_3,P^*_4,P^*_5)$}%
}}}}
\put(5800,-3204){\makebox(0,0)[lb]{\smash{{\SetFigFont{12}{14.4}{\rmdefault}{\mddefault}{\itdefault}{\color[rgb]{0,0,0}$P^*_1$}%
}}}}
\put(5550,-2971){\makebox(0,0)[lb]{\smash{{\SetFigFont{12}{14.4}{\rmdefault}{\mddefault}{\itdefault}{\color[rgb]{0,0,0}$P^*_2$}%
}}}}
\put(5329,-2716){\makebox(0,0)[lb]{\smash{{\SetFigFont{12}{14.4}{\rmdefault}{\mddefault}{\itdefault}{\color[rgb]{0,0,0}$P^*_3$}%
}}}}
\put(5089,-2491){\makebox(0,0)[lb]{\smash{{\SetFigFont{12}{14.4}{\rmdefault}{\mddefault}{\itdefault}{\color[rgb]{0,0,0}$P^*_4$}%
}}}}
\end{picture}}
\caption{Minimal tiling and its corresponding path family.}
\label{mintiling}
\end{figure}

We assign the dominoes in the Aztec rectangle with holes as follows\footnote{This weight assignment was introduced in \cite{Kamioka}.}.  Each even horizontal and odd vertical domino a weight $1$, each odd horizontal domino on level $k$ from the bottom of the region  a weight $tq^{2k}$, and each even vertical domino on $k$ a weight $q^{2k+1}$ (see Figure \ref{elementmove} for four types of dominoes). Similar to the case of rhombi-paths in the previous section, we define the weight $\wt(P_i)$ of the path $P_i$ to be the product of weights of all dominoes corresponding to the steps in $P_i$; and $\wt(\mathbf{P})=\prod_{i=1}^{m}\wt(P_i)$. The weight of each tiling of the region can be written as a product of the form $t^{x}q^{y}$. Denote by $\beta(\textbf{P})$ the exponent $y$ of $q$ in the weight $\wt(\textbf{P})=\wt(T)$. We denote by $\level(P_i)$, $\down(P_i)$ and $\up(P_i)$ the numbers of level, down and up steps in the path $P_i$, respectively. Define $\level(\textbf{P}):=\sum_{i=1}^{m} \level(P_i)$. 

\begin{lem}\label{relate}
Assume that $T$ is a tiling of the region $\mathcal{AR}_{m,n}(s_1,s_2,\dotsc,s_{m})$, and $\textbf{P}=(P_1,P_2,\dotsc,P_m)$ is  the family of non-intersecting (partial) Schr\"{o}der paths corresponding to $T$. Then
\begin{equation}\label{relate1}
v(T)+\level(\textbf{P})=\frac{m(m+1)}{2},
\end{equation}
and
\begin{equation}\label{relate2}
\beta(\textbf{P})-r(T)=\beta(\textbf{P}^*)=\sum_{1\leq j\leq m}2(s_{i}+j-i-1),
\end{equation}
where $\textbf{P}^*$ is the path family corresponding to the minimal tiling $T_0$.
\end{lem}
\begin{proof}
It is easy to see that
\begin{equation}\label{neweq1}
\up(P_i)-\down(P_i)=s_i-i.
\end{equation}
Thus, by adding $s_i-i$ down steps to the right of $P_i$, we have a shifted Schr\"{o}der path $P'_i$ connecting $(-i,i)$ and $(i+2(s_i-i),i)$. One readily sees that
\begin{equation}\label{neweq2}
\down(P'_i)=\up(P'_i)=\up(P_i),
\end{equation}
\begin{equation}\label{neweq3}
\level(P'_i)=\level(P_i),
\end{equation}
and
\begin{equation}\label{neweq4}
\down(P'_i)=\down(P_i)+(s_i-i).
\end{equation}
Moreover, we have also
\begin{equation}
\up(P'_i)+\down(P'_i)+2\level(P'_i)=2s_i,
\end{equation}
so by (\ref{neweq2}) 
\begin{equation}\label{neweq5}
\down(P'_i)+\level(P'_i)=s_i.
\end{equation}
Substituting (\ref{neweq3}) and (\ref{neweq4}) into (\ref{neweq5}), we have
\begin{equation}\label{relate3}
\down(P_i)+\level(P_i)=i,
\end{equation}
for any $i=1,2,\dots,m$. Adding  $m$ equalities in (\ref{relate3}), for $i=1,2,\dotsc,m$, we obtain (\ref{relate1}).

\begin{figure}\centering
\includegraphics{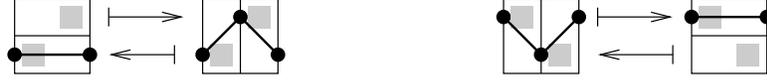}
\caption{The elementary moves rise the rank of the tiling $T$ by one (left-to-right, respectively) if only if the exponent of $q$ in $w(T)$ increases by one.}
\label{elementmove2}
\end{figure}

\medskip

Divide the set of elementary moves into to types as in Figure \ref{elementmove2}. We notice that the elementary moves (from left to right) increase simultaneously the rank $r(T)$ and the exponent $\beta(\mathbf{P})$ by one. This implies the first equality in (\ref{relate2}), since $r(T_0)=0$.

Consider the paths $P^*_j$ in the path family $\mathbf{P}^*$ corresponding to the minimal tiling $T_0$. The path $P^*_{j+1}$ can be obtained recursively from $P^*_{j}$ by adding $s_{j+1}-s_{j}$ up steps followed by a level step, and shifting the resulting path $\sqrt{2}$ units to northwest (see Figure \ref{mintiling}). 

We note that all up steps in $P^*_{j+1}$ have weight $1$. Moreover, the above shifting rises simultaneously the weights of all $j$ level steps in $P^{*}_{j}$ by a factor $q^2$; and the weight of the last level step in $P^{*}_{j+1}$ is $tq^{s_{j+1}-1}$. Thus, we get
\begin{equation}
\wt(P^*_{j+1})=q^{2j}tq^{2(s_{j+1}-1)}\wt(P^*_j),
\end{equation}
and it is easy to see that $\wt(P_1^*)=tq^{2(s_1-1)}$.
 Thus, by induction, we get 
 \begin{equation}
 \wt(P^*_j)=\prod_{i=1}^{j}tq^{2(s_{i}+j-i-1)},\end{equation}
 for $j=1,2,\dotsc,m$. Therefore,
\begin{equation}
\wt(T_0)=\wt(\textbf{P}^*)=t^{\frac{n(n+1)}{2}}q^{\sum_{1\leq i\leq j\leq m}2(s_{i}+j-i-1)},
\end{equation}
which implies the second equal sign in (\ref{relate2}).
\end{proof}

We are now ready to prove Theorem \ref{main}.

\begin{proof}[Proof of Theorem \ref{main}] Weighting the dominoes in our region as in Lemma \ref{relate}.
We use the shorthand notation $AR(t,q)$ for the weighted Aztec rectangle graph $AR_{m,n}\left(\wt_{t,q}^{1,1}(q^2)\right)$ in which all bottommost vertices, except for the ones at the positions $s_i$'s, have been removed. Thus, $AR(t,q)$ is simply the dual graph of our weighted region. By the above bijection between domino tilings and families of non-intersecting Sch\"{o}der paths, we obtain
\begin{equation}
\M(AR(t,q))=\sum_{\textbf{P}=(P_1,\dotsc,P_m)}\wt(\textbf{P})=\sum_{\textbf{P}=(P_1,\dotsc,P_m)}t^{\level(\textbf{P})}q^{\beta(\textbf{P})}.
\end{equation}
By Lemma \ref{relate}, we have
\begin{align}
\sum_{T}t^{v(T)}q^{r(T)}
&=q^{-\sum_{1\leq i\leq j  \leq m}2(s_{i}+j-i-1)}t^{m(m+1)/2}\sum_{\textbf{P}=(P_1,\dotsc,P_m)}t^{-level(\textbf{P})}q^{\beta(\textbf{P})}\\
&=q^{-\sum_{1\leq i\leq j  \leq m}2(s_{i}+j-i-1)}t^{m(m+1)/2}\M(AR(t^{-1},q)),
\end{align}
and the theorem follows from Theorem \ref{weightrectangle} by letting $a=b=1$, $c=t^{-1}$, $d=q$, and replacing $q$ by $q^2$.
\end{proof}

\thispagestyle{headings}

\end{document}